\documentclass[10pt]{article}
\usepackage{bqf-paper}
\usepackage{amsthm}
\usepackage{nameref}
\usepackage{appendix}
\usepackage{titling}

\def\ZZ{\mathbb{Z}}
\def\QQ{\mathbb{Q}}
\def\RR{\mathbb{R}}
\def\CC{\mathbb{C}}

\definecolor{codegreen}{rgb}{0,0.6,0}
\definecolor{codegray}{rgb}{0.5,0.5,0.5}
\definecolor{codepurple}{rgb}{0.58,0,0.82}
\definecolor{backcolour}{rgb}{0.95,0.95,0.92}

\newcommand\blfootnote[1]{
  \begingroup
  \renewcommand\thefootnote{}\footnote{#1}
  \addtocounter{footnote}{-1}
  \endgroup
}

\lstdefinestyle{mystyle}{
    backgroundcolor=\color{backcolour},   
    commentstyle=\color{codegreen},
    keywordstyle=\color{magenta},
    numberstyle=\tiny\color{codegray},
    stringstyle=\color{codepurple},
    basicstyle=\ttfamily\footnotesize,
    breakatwhitespace=false,         
    breaklines=true,                 
    captionpos=b,                    
    keepspaces=true,                 
    numbers=left,                    
    numbersep=5pt,                  
    showspaces=false,                
    showstringspaces=false,
    showtabs=false,                  
    tabsize=2
}

\newcommand{\kset}{K}

\title{\textbf{On Three-Term Linear Relations for Theta Series of Positive-Definite Binary Quadratic Forms}}
\author{
  Rahul Saha* and 
  Jonathan Hanke* \\
}
\date{\today}
\begin{document}

\maketitle
\blfootnote{* Differentiated contributions as first author and senior author respectively. Rahul Saha implemented the extended refinement algorithm in \texttt{SageMath}, established \autoref{lemma:linear} and \autoref{lemma:non_trivial_relation_theorem}, and handled the $a+b \ge 4$ case. Jonathan Hanke introduced the extension of Schiemann's method to determine all theta series satisfying a given linear relation, and supervised the first author's Princeton undergraduate junior paper and reading courses on this topic. This paper was jointly written, with numerous unattributed comments between the authors to improve the clarity of exposition.}
\vspace{-30pt}
\abstract{
    In this paper, we investigate three-term linear relations among theta series of positive-definite integral binary quadratic forms. We extend Schiemann's methods to characterize all possible three-term linear relations among theta series of such forms, providing necessary and sufficient conditions for such relations to exist. To accomplish this, we develop, implement, and execute a novel extended refinement algorithm on polyhedral cones. We show that there is exactly one non-trivial three-term linear relation: it involves quadratic forms with discriminants $-3, -12, -48$, all in the same rational squareclass $-3(\mathbb{Q}^\times)^2$.

\vspace{-15pt}

\setcounter{tocdepth}{2}
\tableofcontents
\addtocontents{toc}{\vspace{-11pt}}
\newpage

\section{Introduction}
\label{sec:introduction}

Let $Q$ be a positive-definite integer-valued binary quadratic form, which we denote by
$$
Q(\vec{x}) := Q(x,y) := ax^2 + bxy + cy^2 
$$
where $a, b, c \in \ZZ$ and $\mathrm{Disc}(Q) := \Delta(Q) := b^2 - 4ac < 0$. For any non-negative integer $m \in \Z_{\ge 0}$, we consider the sets 
$$
\rep_Q(m) := \{ \vec{x} \in 
\ZZ^{\oplus 2} 
\mid 
Q(\vec{x}) = m\}
$$
of representation vectors $\vec{x} := (x, y)$, whose cardinality defines the {\bf representation numbers} 
$$
r_Q(m) := \#\rep_Q(m).
$$  
We note that the representation numbers $r_Q(m)$ are unchanged when we perform any 
$\ZZ$-invertible linear change of variables $\vec{x}' := A\vec{x}$ (i.e. where both $A$ and $A^{-1} \in \mathrm{GL}_{d=2}(\ZZ)$)
on $Q(\vec{x})$ to obtain another {\bf $\ZZ$-equivalent} quadratic form $Q'(\vec{x}')$, 
and we denote this $\ZZ$-equivalence by $Q \sim_{\ZZ} Q'$.

The representation numbers $r_{Q}(m)$ of a positive-definite integer-valued quadratic form $Q$ collectively define its {\bf theta series}
$$
\uptheta_Q \defeq \uptheta_Q(z) \defeq \sum\limits_{m \ge 0} r_{Q}(m) q^m,
$$
where $q:= e^{2\pi iz}$ and $z:= x + iy\in \CC$ with $y > 0$. The theta series $\uptheta_Q(z)$ encodes much of the arithmetic of the quadratic form $Q$, and determining its properties has been a central question of Number Theory and many other related areas of mathematics.

\smallskip
One such question  about the spectrum of the Laplacian on a Riemannian manifold $M$ popularized by Kac in the 1960s is ``Can you hear the shape of a drum?'', which asks if one can determine the underlying manifold from its Laplacian eigenvalues (with multiplicity).  
In the special case where $M$ is a flat torus of dimension $d$ (i.e. $M = \RR^d / \Lambda$ for some full rank lattice $\Lambda \subset \RR^d$), this is equivalent to asking 
$$
\begin{matrix}
\text{``Can we determine (the $\ZZ$-equivalence class of) a positive-definite integer-valued }
\\ 
\text{quadratic form $Q$ in $d$ variables from its theta series $\uptheta_Q(z)$?''}
\end{matrix}
$$
This was definitively answered by Schiemann in the 1990s using a novel computational approach of polyhedral cone decompositions, 
where he defined and executed an algorithm to answer this question.  Using this approach, he showed that the 
uniqueness result is true when $d = 3$, 
and gave an example when $d=4$ where uniqueness fails.  (See \cite{MR4520776} for a recent survey of this topic.)

\medskip
In this paper we propose to view Schiemann's results from the perspective of {\it linear relations} among theta series $\uptheta_{Q}(z)$ of quadratic forms, and observe that in this context he answers the question ``Are there any {\it non-trivial} 2-term linear relations 
\begin{equation}
\label{eqn:two_term_linear}
\alpha_1 \uptheta_{Q_1} + \alpha_2 \uptheta_{Q_2}  = 0
\end{equation}
among the theta series of positive-definite integer-valued quadratic forms of dimension $d$?''. (No when $d \leq 3$, and yes when $d \geq 4$, where necessarily $\alpha_1 = -\alpha_2$.) We then supplement and substantially extend his polyhedral cone approach to determine \underline{all} (trivial and non-trivial) \underline{3-term} linear relations 
\begin{equation}
\label{eqn:three_term_linear}
\alpha_1 \uptheta_{Q_1} + \alpha_2 \uptheta_{Q_2} + \alpha_3 \uptheta_{Q_3}  = 0
\end{equation}
that exist between theta series of positive-definite integer-valued binary quadratic forms (i.e. $d=2$), with our main result being:

\begin{theorem}
\label{thm:theorem_1_0}
Suppose $Q_1, Q_2, Q_3$ are positive-definite integer-valued binary quadratic forms satisfying the 3-term linear relation 
$$
\alpha_1 \uptheta_{Q_1} + \alpha_2 \uptheta_{Q_2} + \alpha_3 \uptheta_{Q_3}  = 0
$$
for some $\alpha_1, \alpha_2, \alpha_3 \in \RR$.  Then, up to reordering of indices $i \in \{1,2,3\}$,  
exactly one of the following holds:
\begin{enumerate}
\item[1)] {\bf (Degenerate Solutions)} All $\alpha_i = 0$, and all $Q_i$ are arbitrary.
\item[2)] {\bf (Trivial 2-term Solutions)} $\alpha_1 = -\alpha_2 \neq 0$, $\alpha_3 = 0$,   $Q_1 \sim_\ZZ Q_2$, and $Q_3$ is arbitrary.
\item[3)] {\bf (Trivial 3-term Solutions)} All $\alpha_i \neq 0$, $\alpha_1 + \alpha_2 + \alpha_3 = 0$, and  $Q_1 \sim_\ZZ Q_2 \sim_\ZZ Q_3$.
\item[4)] {\bf (Non-Trivial 3-term Solutions)}  $6 \alpha_1 = 3\alpha_2 = -2\alpha_3 \neq 0$, giving the relation
\begin{equation}
\label{eqn:non_trivial_eqn}
\frac{1}{3} \uptheta_{Q_1} + \frac{2}{3} \uptheta_{Q_2} = \uptheta_{Q_3},
\end{equation}
where 
$$
Q_1 \sim_\ZZ c(x^2 + xy + y^2), 
\qquad
Q_2 \sim_\ZZ 4c(x^2 + xy + y^2),
\qquad
Q_3 \sim_\ZZ c(x^2 + 3y^2),
$$
for some positive integer $c \in \ZZ_{>0}.$
\end{enumerate}
\end{theorem}

\section{The \texorpdfstring{$\mathrm{GL}_2(\ZZ)$}{GL2(Z)} Reduction Domain \texorpdfstring{$\mathcal{V}$}{\mathcal{V}}}

In this section we define the (usual) fundamental domain $\mathcal{V}$ for the action of $\mathrm{GL}_2(\ZZ)$ on real-valued 
positive-definite binary quadratic forms, and give its explicit description as a polyhedral cone.

\medskip

\begin{convention}[\textbf{Binary Quadratic Forms as Tuples}] 
\label{convention:BQFs_as_tuples}
Throughout this paper, we adopt the convention of identifying real-valued binary quadratic forms with $\R^3$ via the bijection
$$
Q(x, y) \defeq q_{11} x^2 + q_{12} xy + q_{22} y^2 
\mapsto
(q_{11}, q_{22}, q_{12}) \in \RR^3.
$$
We also identify $Q$ with the symmetric matrix $\left[\begin{smallmatrix}
    q_{11} & q_{12}/2 \\ 
    q_{12}/2 & q_{22} \\ 
\end{smallmatrix}\right]$, noting that \[Q(x, y) = \begin{bmatrix}
    x & y \\
\end{bmatrix} \begin{bmatrix}
    q_{11} & q_{12}/2 \\ 
    q_{12}/2 & q_{22} \\ 
\end{bmatrix} \begin{bmatrix}
    x \\ 
    y \\ 
\end{bmatrix}. \]  
\end{convention}

\medskip
  
\begin{definition}[\textbf{$\mathrm{GL}_2(\ZZ)$ Reduction Domain}] 
\label{def:reduced_forms}
Using the identification in \autoref{convention:BQFs_as_tuples}, we define the set $\mathcal{V} \subset \R^3$ of $\mathrm{GL}_2(\ZZ)$-\textbf{reduced positive-definite binary quadratic forms} by 
\[(q_{11}, q_{22}, q_{12}) \in \mathcal{V} \iff \begin{cases} 
q_{22} - q_{11} \ge 0,  \\ 
q_{12} \ge 0, \\ 
q_{11} - q_{12} \ge 0, \text{ and }\\ 
q_{11} > 0, \\
\end{cases} \]
and also let $\overline{\mathcal{V}}$ denote closure of $\mathcal{V}$.  
We say that a positive-definite binary quadratic form $Q(x,y)$ is {\bf reduced} if the corresponding vector $(q_{11}, q_{22}, q_{12})$ 
is in $\mathcal{V}$.
\end{definition}

\begin{lemma}[\textbf{Strong Fundamental Domain}]
  Given any real positive-definite binary quadratic form $Q$, there exists a unique (reduced) binary quadratic form 
  $Q' \in \mathcal{V}$ so that $Q \sim_\ZZ Q'$. 
\end{lemma}

\begin{proof}
  We define the (usual) reduction domain $\mathcal{V}' \subset \RR^3$ of $\mathrm{SL}_2(\ZZ)$-reduced positive-definite binary quadratic forms as follows: 
  \[(q_{11}, q_{22}, q_{12}) \in \mathcal{V}' \iff \begin{cases} 
 q_{22} \ge q_{11} \ge |q_{12}|, \\ 
 q_{11} > 0, \\  
 q_{11} = q_{22} \implies q_{12} \ge 0, \text{ and } \\ 
 q_{11} = q_{12} \implies q_{12} \ge 0. \\ 
\end{cases} \]

By \cite[\S2]{buell1989binary}, given any real positive-definite binary quadratic form $Q$, there exists $Q' \in \mathcal{V}'$ so that $Q = X^T Q' X$ for some $X \in \mathrm{SL}_2(\ZZ)$. Furthermore we observe that \[\begin{bmatrix}
    1 & 0 \\ 
    0 & -1 \\ 
\end{bmatrix} \begin{bmatrix}
    q_{11} & q_{12}/2 \\ 
    q_{12}/2 & q_{22} \\ 
\end{bmatrix} \begin{bmatrix}
    1 & 0 \\ 
    0 & -1 \\ 
\end{bmatrix} = \begin{bmatrix}
    q_{11} & -q_{12}/2 \\ 
    -q_{12}/2 & q_{22} \\ 
\end{bmatrix}  \] 
which implies that $(q_{11}, q_{22}, q_{12}) \sim_\ZZ (q_{11}, q_{22}, -q_{12})$, and they are related by a linear transform $Y \defeq \left[\begin{smallmatrix}
    1 & 0 \\ 
    0 & -1 \\ 
\end{smallmatrix}\right]$ with determinant $-1$. Since $\det(\mathrm{GL}_2(\ZZ)) = \{\pm 1\}$, we also know $g \in \mathrm{GL}_2(\ZZ)$ has either $g \in \mathrm{SL}_2(\ZZ)$ or $gY = gY^{-1} \in \mathrm{SL}_2(\ZZ)$, giving the disjoint union
\[\mathrm{GL}_2(\ZZ) = \mathrm{SL}_2(\ZZ) \:\: \sqcup \:\: Y \cdot \mathrm{SL_2}(\ZZ).  
\] 
 Thus for any positive-definite binary quadratic form $Q$, there must exist $Q' \in \mathcal{V}'  \intersect \{(q_{11}, q_{22}, q_{12}) \in \R^3 \mid q_{12} \ge 0\}$ so that $Q \sim_\ZZ Q'$. To see $Q'$ is unique, assume that there exists another $Q'' \in \mathcal{V}' \{(q_{11}, q_{22}, q_{12}) \in \R^3 \mid q_{12} \ge 0\}$ so that $Q \sim_\ZZ Q''$. Then $Q' \sim_{\ZZ} Q''$. By uniqueness in \cite[\S2]{buell1989binary}, they cannot be related by an element of $\mathrm{SL}_2(\ZZ)$. Therefore they must be related by an element of $Y\cdot\mathrm{SL}_2(\ZZ)$, however then $Y^TQ'Y \sim_\ZZ Q''$, which also violates uniqueness in \cite[\S2]{buell1989binary}. Therefore $\mathcal{V}' \intersect \{(q_{11}, q_{22}, q_{12}) \in \R^3 \mid q_{12} \ge 0\} = \mathcal{V}$, as desired.   
\end{proof}

\smallskip
\noindent
We note that $\mathcal{V}$ can be thought of as a polyhedral cone, in the following sense:
\begin{definition}[\textbf{Polyhedral Cone $\cP(A, B)$}]
\label{def:polycone}
Suppose that $A$ and $B$ are finite subsets of $\R^n$.  Then we define the \textbf{polyhedral cone} $\cP(A, B) \subseteq \R^n$ as 

    \[\cP(A, B) \defeq \{\vx \in \R^n \mid 
    \vec{a} \cdot \vx \ge 0  \emph{ for all } \: \vec{a} \in A, 
    \emph{ and }
    \vec{b} \cdot \vx > 0 \emph{ for all }  \: \vec{b} \in B
    \}.\]    
\end{definition} 
\begin{remark}[\textbf{$\mathcal{V}$ as a Polyhedral Cone}]
\label{rem:polyconeV}
Using the notation for polyhedral cones, we can write 
$\mathcal{V}\defeq\cP(\mathbf{A}, \mathbf{B}),$
where 
$\mathbf{A} \defeq \{(-1, 1, 0), (1, 0, -1), (0, 0, 1)\}$ and  $\mathbf{B} \defeq \{(1, 0, 0)\}$.
\end{remark} 
\begin{remark}[\textbf{Edges of $\mathcal{V}$}]
\label{rem:edges_of_V}
    We can compute that the edges of $\mathcal{V}$ are given by $(0, 1, 0), (1, 1, 0),$ and $(1, 1, 1)$, and let $\cE_1, \cE_2, \cE_3$ denote the associated quadratic forms under \autoref{convention:BQFs_as_tuples}, that is, \[\cE_1(x, y) \defeq y^2, \quad \cE_2(x, y) \defeq x^2 + y^2, \quad \text{ and } \quad \cE_3(x, y) \defeq x^2 + xy + y^2.\] 
\end{remark}

\section{The Partial Ordering on Strongly Primitive Vectors}
\label{sec:pre_order}

In this section we define a relation $\preceq$ on $\Z^{\oplus 2}$ induced by $\mathcal{V}$, and the notion of a minimal subset. 

\subsection{Strongly Primitive Vectors}

In this subsection we define the notion of strongly primitive objects in several related contexts. 

\medskip
\begin{definition}[\textbf{Strongly Primitive Objects}]
\label{def:strongly_primitive_objects}
Let $\Z^{\oplus 2}_{\tilde{*}}$ be the space of all \textbf{strongly primitive vectors} $(x, y)$ so that $\gcd(x, y) = 1$ and the last non-zero coordinate is positive. For $m \in \Z_{\ge 0}$, the \textbf{strongly primitive representation set} $\rep^{\tilde{*}}_Q(m)$ of $m$ is defined as the set of integral solutions $(x, y) \in \Z^{\oplus 2}_{\tilde{*}}$ to $Q(x, y) = m$. Then the \textbf{strongly primitive representation number} $r^{\tilde{*}}_Q(m)$ is defined as the number of elements in $\rep^{\tilde{*}}_Q(m)$. This lets us define the \textbf{strongly primitive theta series} $\uptheta^{\tilde{*}}_Q$ as 
\[
\uptheta^{\tilde{*}}_Q \defeq \uptheta^{\tilde{*}}_Q(z) \defeq \sum\limits_{m \ge 0} r^{\tilde{*}}_{Q}(m) q^m.
\] 
\end{definition}

\begin{remark}[\textbf{Primitive Objects}] 
In the literature, it is common to define the set of primitive vectors $\ZZ^{\oplus 2}_*$ as the space of all $(x,y) \in \ZZ^{\oplus 2}$ with $\gcd(x, y) = 1$. The associated primitive objects $\Z^{\oplus 2}_*, \rep_Q^*(m), r^*_Q(m),$ and $\uptheta^*_Q(q)$ are defined in terms of these primitive vectors. The relationship between primitive and strongly primitive representation numbers is given by  $r^*_Q(m) = 2r^{\tilde{*}}_Q(m)$ for $m \ge 0$. Note that $\vec{0} = (0, 0)$ is neither a primitive nor a strongly primitive vector, since $\gcd (0, 0) \neq 1$. 
\end{remark}

 The next lemma allows us to pass from linear relations among theta series to linear relations among strongly primitive theta series. Note that since $\uptheta_Q$ always has constant term $1$ for a positive-definite quadratic form $Q$, it forces the coefficients $\alpha_i$ in \autoref{eqn:three_term_linear} to satisfy the relation $\alpha_1 + \alpha_2 + \alpha_3 = 0$.
\begin{lemma}[\textbf{Relating $\uptheta_{Q}$ and $\uptheta^{\tilde{*}}_{Q}$}]
Suppose $Q_1, Q_2, Q_3$ are positive-definite integer-valued binary quadratic forms and $\alpha_1, \alpha_2, \alpha_3 \in \R$ are real numbers satisfying $\alpha_1 + \alpha_2 + \alpha_3 = 0$. Then \[\alpha_1 \uptheta_{Q_1} + \alpha_2 \uptheta_{Q_2} + \alpha_3 \uptheta_{Q_3} = 0 \iff \alpha_1 \uptheta^{\tilde{*}}_{Q_1} + \alpha_2 \uptheta^{\tilde{*}}_{Q_2} + \alpha_3 \uptheta^{\tilde{*}}_{Q_3} = 0.\]
\end{lemma} 

\begin{proof} \noindent \textbf{Case 1 ($m \ge 1$).} Suppose $m \ge 1$. Then we have, \[r_{Q_i}(m) = \sum\limits_{d | m} 2r^{\tilde{*}}_{Q_i}(d) \]
    for $i \in \{1, 2, 3\}$, giving 
    \begin{equation*}
        \begin{split}
           \alpha_1 r_{Q_1}(m) + \alpha_2 r_{Q_2}(m) + \alpha_3 r_{Q_3}(m) &=  \alpha_1\sum\limits_{d | m} 2r^{\tilde{*}}_{Q_1}(d)  +  \alpha_2\sum\limits_{d | m} 2r^{\tilde{*}}_{Q_2}(d) +  \alpha_3\sum\limits_{d | m} 2r^{\tilde{*}}_{Q_3}(d) \\
           &= \sum\limits_{d | m} 2(\alpha_1 r^{\tilde{*}}_{Q_1}(d)+ \alpha_2 r^{\tilde{*}}_{Q_2}(d)+ \alpha_3 r^{\tilde{*}}_{Q_3}(d)) \\
           &= \sum\limits_{d | m} 0 \\ 
           &= 0. \\
        \end{split}
    \end{equation*}
Conversely, by M\"{o}bius inversion we have
\[r^{\tilde{*}}_Q(m) = \frac{1}{2} \sum\limits_{m = dd'} \mu(d) r_Q(d')\]
where the M\"{o}bius function $\mu$ is defined by 
\[\mu(m) \defeq \begin{cases}
    +1 \quad \text{ if }  m \text{ is a square-free positive integer with an even number of prime factors,} \\
    -1 \quad \text{ if }  m \text{ is a square-free positive integer with an odd number of prime factors,} \\
    \text{ } 0 \quad \text{\:\: if }  m \text{ is not square-free.} \\
\end{cases}\]
Therefore 
    \begin{equation*}
        \begin{split}
           \alpha_1 r^{\tilde{*}}_{Q_1}(m) + \alpha_2 r^{\tilde{*}}_{Q_2}(m) + \alpha_3 r^{\tilde{*}}_{Q_3}(m) &=  \frac{1}{2} \left(\alpha_1\sum\limits_{m = dd'} \mu(d)r_{Q_1}\left(d'\right)  +  \alpha_2\sum\limits_{m = d d'}\mu(d) r_{Q_2}\left(d'\right) +  \alpha_3\sum\limits_{m = dd'} \mu(d)r_{Q_3}\left(d'\right)\right) \\
           &= \frac{1}{2}\sum\limits_{m = dd'} \mu\left(d\right) (\alpha_1 r_{Q_1}(d')+ \alpha_2 r_{Q_2}(d')+ \alpha_3 r_{Q_3}(d')) \\
           &= \frac{1}{2}\sum\limits_{m = dd'} 0 \\ 
           &= 0. \\
        \end{split}
    \end{equation*}

  \textbf{Case 2 ($m = 0$)}. Suppose $m = 0$. Then  $\alpha_1 r_{Q_1}(m) + \alpha_2 r_{Q_2}(m) + \alpha_3 r_{Q_3}(m) = \alpha_1 + \alpha_2 + \alpha_3 = 0$. Conversely, $\alpha_1 r^{\tilde{*}}_{Q_1}(m) + \alpha_2 r^{\tilde{*}}_{Q_2}(m) + \alpha_3 r^{\tilde{*}}_{Q_3}(m) = \alpha_1 \cdot 0 + \alpha_2 \cdot 0 + \alpha_3 \cdot 0 = 0$, as desired. 
\end{proof}

As a consequence of this lemma, we can focus on strongly primitive vectors and strongly primitive representation numbers throughout the rest of the paper, simplifying complexity of our computations. 

\medskip 

\subsection{The Partial Ordering} 

In this subsection we define the relation $\preceq$ and give an explicit algorithm for determining it. 

\begin{definition}[\textbf{The Relations $\preceq$ and $\succeq$}]
\label{def:pre_order}
    Suppose $\vx, \vy \in \Z^{\oplus 2}$. Then we define the relation    
    \[\vx \preceq \vy \iff Q(\vx) \le Q(\vy) \emph{ for all } Q \in \mathcal{V}. \] 
    We also define the opposite relation $\succeq$ by $\vy \succeq \vx \iff \vx \preceq \vy$. 
\end{definition}

The next lemma gives us an explicit algorithm for determining when the relation $\vx \preceq \vy$ holds.

\begin{lemma}[\textbf{Determining $\vx \preceq \vy$}]
\label{lemma:preceq_algo}
     Suppose $\vx, \vy\in \ZZ^{\oplus 2}$. Then $\vx \preceq \vy$ if and only if $\cE_i(\vx) \le \cE_i(\vy)$ for each $i \in \{1, 2, 3\}$, 
     where $\cE_i$ are quadratic forms as defined in \autoref{rem:edges_of_V}.
\end{lemma}
\begin{proof}  
  \textbf{($\Longleftarrow$)} Assume $\cE_i(\vx) \le \cE_i(\vy)$ for each $i=1, 2, 3$. Since $\mathcal{V} \subset \overline{\mathcal{V}}$, any form $Q \in \mathcal{V}$ can be written as a $\R_{\ge 0}$-linear combination of the edges, i.e. $Q = \lambda_1 \cE_1 + \lambda_2 \cE_2 + \lambda_3 \cE_3$ for $\lambda_i \in \R_{\ge 0}$. Therefore
\begin{equation*}
    \begin{split}
          Q(\vx) = \vx^T Q \vx &= \sum\limits_{j=1}^3 \lambda_j \vx^T \cE_j \vx  
                 \le \sum\limits_{j=1}^3 \lambda_j \vy^T \cE_j \vy = \vy^T Q \vy =  Q(\vy) \\
    \end{split}
\end{equation*}
which proves one direction.

 \textbf{($\Longrightarrow$)} Assume  $\vx \preceq \vy$ (i.e. $Q(\vx) \le Q(\vy)$ for all $Q \in \mathcal{V}$). First note that $\cE_2$ and $\cE_3$ are contained  in $\mathcal{V}$, so it suffices to prove that $\cE_1(\vx) \le \cE_1(\vy)$. Since $\cE_1$ lies on the boundary of $\mathcal{V}$, we can construct a sequence of forms $\cE^{(1)}, \cE^{(2)}, \cE^{(3)}, \ldots \in \mathcal{V}$ whose limit $\lim_{i \to \infty} Q^{(i)} = \cE_1$. Since each $Q^{(i)} (\vx) \le Q^{(i)}(\vy)$, we must have $\cE_1(\vx) \le \cE_1(\vy)$, as desired.   
\end{proof}

\begin{remark}[\textbf{Failure of Antisymmetry on $\Z^{\oplus 2}$}]
 Note that the relation $\preceq$ is not a partial ordering on $\ZZ^{\oplus 2}$, since $\vx \preceq -\vx$ and $-\vx \preceq \vx$, violating antisymmetry of a partial ordering. However, it turns out  the relation $\preceq$ is a partial ordering on the set $\ZZ^{\oplus 2}_{\tilde{*}}$ of strongly primitive vectors, as we show in the next lemma.     
\end{remark}

\smallskip
\begin{lemma}[\textbf{Partial Ordering on $\ZZ^{\oplus 2}_{\tilde{*}}$}]
    \label{lemma:preceq_partial_ordering_on_strongly_primitive_vectors}
    The relation $\preceq$ is a partial ordering on $\ZZ^{\oplus 2}_{\tilde{*}}$. 
\end{lemma}

\begin{proof}
    We verify reflexivity, antisymmetry, and transitivity of $\preceq$ on $\ZZ^{\oplus 2}_{\tilde{*}}$. Below, we suppose $\vx \defeq (x_1, x_2), \vy \defeq (y_1, y_2), \vec{z} \defeq (z_1, z_2) \in \ZZ^{\oplus 2}_{\tilde{*}}$. 
    \begin{enumerate}
        \item[1)]  \textbf{Reflexivity}. $\vx \preceq \vx$ since $Q(\vx) = Q(\vx)$ for all $Q \in \mathcal{V}$. 

        \item[2)] \textbf{Antisymmetry}. Suppose $\vx \preceq \vy$ and $\vy \preceq \vx$, that implies $Q(\vx) \le Q(\vy)$ and $Q(\vy) \le Q(\vx)$ for all $Q \in \mathcal{V}$, which implies $Q(\vx) = Q(\vy)$ for all $Q \in \mathcal{V}$. By \autoref{lemma:preceq_algo}, this implies $\cE_i(\vx) = \cE_i(\vy)$ for each $i \in \{1,2,3\}$, giving 
            \[
            x_2^2 = y_2^2, \qquad  
            x_1^2 + x_2^2 = y_1^2 + y_2^2, \quad \mathrm{ and } \quad  
            x_1^2 + x_1x_2 + x_2^2 = y_1^2 + y_1 y_2 + y_2^2.\]
         From the first equation, we get $x_2 = \pm y_2$. From the second equation, we get $x_1 = \pm y_1$. From the third equation, we get $x_1x_2 = y_1y_2$ which implies that either $(x_1, x_2) = (y_1, y_2)$ or $(x_1, x_2) = (-y_1, -y_2)$. However since $(y_1, y_2)$ is strongly primitive, we know that $(-y_1, -y_2)$ cannot be strongly primitive, giving $\vx = \vy$ as desired. 

         \item[3)] \textbf{Transitivity}. This follows from \autoref{def:pre_order} since $Q(\vx) \le Q(\vy)$ and $Q(\vy) \le Q(\vec{z})$ for all $Q \in \mathcal{V}$ implies $Q(\vx) \le Q(\vec{z})$ for all $Q \in \mathcal{V}$, hence $\vx \preceq \vec{z}$.  
    \end{enumerate}
\end{proof}

\subsection{The Minimal Subset}

In this subsection we define the notion of a minimal subset, and give some useful properties.

\begin{definition}[\textbf{Minimal Subset $\minimum(X)$}]
\label{def:minimum}
    Suppose $X \subseteq \Z^{\oplus 2}$. Then we define the \textbf{minimal subset $\minimum(X)$ of $X$} as \[\minimum(X) \defeq \{\vx \in X \mid \: \vx \not\succeq \vy \emph{ for all } \vy \in X\setminus \{\vx\}\}.\]
\end{definition}

\medskip
Next we prove a technical lemma and give some useful pro{}perties of $\minimum(X)$ that will allow us to compute it in \autoref{sec:refinement_algo}.

\begin{lemma}
\label{lemma:x_a_finite}
    Suppose $X \subset \ZZ^{\oplus 2}$. Then for any $a \in \ZZ$, the set \[S \defeq \minimum(X) \intersect \{(x, y) \in \ZZ^{\oplus 2}\mid y = a\}\] is finite. 
\end{lemma}

\begin{proof}
    Assume by way of contradiction that $S$ is an infinite set. Then there exist $x_1, x_2 \in \ZZ$ having the same sign so that $(x_1, a), (x_2, a) \in \minimum(X)$, and $|x_1| + |x_2| > |a|$. Without loss of generality we also assume that $x_1 \le x_2$. 

        \textbf{Case 1}: Assume $x_1, x_2 \ge 0$, then we show $(x_1, a) \preceq (x_2, a)$ by verifying the conditions in \autoref{lemma:preceq_algo}: \[\cE_1(x_1, a) \le \cE_1(x_2, a) \iff a^2 \le a^2,\] 
        \[ \cE_2(x_1, a) \le \cE_2(x_2, a) \iff x_1^2 + a^2 \le x_2^2 + a_2^2,\] which holds since  $0 \le x_1 \le x_2$, and 
        \[\cE_3(x_1, a) \le \cE_3(x_2, a) \iff x_1^2 + x_1a + a^2 \le x_2^2 + x_2a + a^2 \iff (a + x_1 + x_2)(x_1 - x_2) \le 0,\] which holds since $x_1 - x_2 \le 0$ and $a + x_1 + x_2 \ge 0$ (since by assumption $|x_1 | + |x_2 | > |a|$ and $x_1, x_2 \ge 0$). Thus $(x_1, a) \preceq (x_2, a) \in \minimum(X)$, which is a contradiction. 

         \textbf{Case 2}: Assume $x_1, x_2 \le 0$, then one can similarly show $(x_2, a) \preceq (x_1, a) \in \minimum(X)$, which is once again a contradiction. This completes our proof.       
\end{proof}

\begin{lemma}[\textbf{Properties of $\minimum(X)$}]
\label{lemma:min_properties}
Suppose $X \subset \Z^{\oplus 2}$. Then the following properties hold: 
\begin{enumerate}
    \item [1)] $\minimum(X)$ is a finite set.
    \item [2)] $\minimum(X) \neq \emptyset$  when $X \subseteq \ZZ^{\oplus 2}_{\tilde{*}}$ is non-empty. 
    \item [3)] If $W \subseteq \mathbb{Z}_{\tilde{*}}^{\oplus 2}$ and there exists some non-empty auxiliary set $Y_W$ satisfying both $ Y_W \subseteq X \subseteq \mathbb{Z}_{\tilde{*}}^{\oplus 2}$ and $W \supseteq\{\vx \in X \mid \vx \nsucceq \vy \emph{ for all } \vy \in Y_W\} \cup Y_W$, then  $\operatorname{MIN}(X)=\operatorname{MIN}(X \cap W)$.
    \item [4)] Suppose $(a, 1) \in X \subseteq \Z^{\oplus 2}_{\tilde{*}}$, and let $W_{0, a} \defeq \{\vx \in \Z^{\oplus 2}_{\tilde{*}} \mid x \not\succeq (a, 1)\} \union \{(a, 1)\}$. Then
    \subitem i) $\minimum(X) = \minimum(X \intersect W_{0, a})$, and
    \subitem ii) $W_{0, a} \subseteq \{\vx \in \Z^{\oplus 2}_{\tilde{*}} \mid \: \lVert\vx\rVert_{\infty} \le \sqrt{2(a^2 + \max(a, 0) + 1)}\}$, where $\lVert (x, y) \rVert_{\infty} \defeq \max \{|x|, |y|\}$.
\end{enumerate}
\end{lemma} 

\begin{proof}
        \textbf{Proof of 1)}: Assume otherwise $\minimum(X)$ is an infinite set. Choose some arbitrary $(x_0, y_0) \in \minimum(X)$. Since $\cE_2$ is a positive-definite form, there exists an infinite subset $S \subseteq \minimum(X)$ so that for any $(x, y) \in S$, we have $\cE_2(x, y) \ge \cE_2(x_0, y_0)$. Similarly, since $\cE_3$ is a positive-definite form, there exists an infinite subset $S' \subseteq S \subseteq \minimum(X)$ so that for any $(x, y) \in S'$, we have $\cE_3(x, y) \ge \cE_3(x_0, y_0)$. If there is a point $(x, y) \in S'$ so that $y^2 \ge y_1^2$, then we have a contradiction since  $(x_0, y_0) \preceq (x, y)$ (by \autoref{lemma:preceq_algo}) and so $(x, y)$ cannot be in $\minimum(X)$. Thus for all $(x, y) \in S'$, we have $y^2 < y_1^2 \implies y < |y_1|$, which implies that $y$ is bounded. But since $S'$ is infinite, there must be some coordinate $a \in \ZZ$ so that there are infinitely many $x$ for which $(x, a) \in S' \subseteq \minimum(X)$, which by \autoref{lemma:x_a_finite} is a contradiction.  

        \smallskip
        \textbf{Proof of 2)}: Assume otherwise. Then for any $\vx \in X \subseteq \ZZ^{\oplus 2}_{\tilde{*}}$, we can find $\vx' \in X\setminus \{\vx\}$ such that $\vx' \preceq \vx$. This lets us construct an infinitely descending chain $\vx_1 \succeq \vx_2 \succeq \ldots$. 

         \quad By antisymmetry of the partial order on the set of strongly primitive vectors, no $\vx$ appears twice in this chain. But by \autoref{lemma:preceq_algo}, this implies that the lengths of the vectors in this chain are also decreasing, but this cannot happen forever since there are finitely many vectors of each length, a contradiction.  
        
        \smallskip
        \textbf{Proof of 3)}:  The proof in \cite[\S5.4, Lemma 5.4.4, p46]{rydell2020three} can be used with only one minor modification, which is that we are working with strongly primitive vectors of dimension $2$ instead of dimension $3$.   
        
        \smallskip
        \textbf{Proof of 4)}:  Part 4i) follows from part 3) by letting $Y_W \defeq \{(a, 1)\}$. For part 4ii), suppose we have $|x|, |y| > \sqrt{2(a^2 + \max(a, 0) + 1)}$, then we can verify that $(x, y) \succeq (a, 1)$ by applying \autoref{lemma:preceq_algo}, and thus $(x, y) \not\in W_{0,a}$.
\end{proof}

\begin{remark}[\textbf{Computing $\minimum$}]
\label{rem:compute_min}
     For a finite set $X \subseteq \Z^{\oplus 2}$, we can use part 4i) of \autoref{lemma:min_properties} to compute $\minimum(\Z^{\oplus 2}_{\tilde{*}}\setminus X)$ since there always exists some vector $(a, 1) \in \Z^{\oplus 2}_{\tilde{*}}\setminus X$. Here, the computation is simpler because $W_{0, a}$ is finite by part 4ii).  
\end{remark}

Next we observe that $\minimum(X)$ satisfies the desirable property that any $\text{GL}_2(\ZZ)$-reduced positive-definite binary quadratic form $Q$ achieves its minimum over $X$ at some element of $\minimum(X)$. 

\begin{lemma}[\textbf{$\minimum$ attains Minimum Values}]
\label{lemma:min_alt_def}
Suppose $X \subseteq \Z^{\oplus 2}$. Then for any $Q \in \mathcal{V}$, we have
\[\min (Q(\minimum(X))) = \min(Q(X)).\]
\end{lemma}

\begin{proof}
Suppose otherwise that there exists some $Q \in \mathcal{V}, \vy \not\in \minimum(X)$ so that $Q(\vy) \le Q(\vx)$ for all $\vx \in X$, and $Q(\vy) < Q(\vx)$ for all $\vx \in \minimum(X)$. But then $\vy$ is a vector such that $\vx \not\preceq \vy$ for all $\vx \in X$, and thus by definition, should be in $\minimum(X)$. This is a contradiction.  
\end{proof}

We now define the different but related notion of successive minima. For two singleton sets $A = \{a\}, B = \{b\}$, we use the notation $A \le B$ to mean $a \le b$.   

\begin{definition}[\textbf{Successive Minima Sequence}]
\label{def:succ_min_set}
Suppose $Q$ is a positive-definite binary quadratic form, and $(X_i)_{i=1}^{\infty}$ is a sequence of finite sets $X_i \subseteq \Z^{\oplus 2}_{\tilde{*}}$.  We define $(X_i)_{i=1}^{\infty}$ to be a \textbf{successive minima sequence} of $Q$ if and only if the following conditions hold:
\begin{enumerate}
    \item[1)] For all $\vx \in \Z^{\oplus 2}_{\tilde{*}}$, there exists a unique $i$ so that $\vx \in X_i$.
    \item[2)] For all $i$, there exists an $m \in \R_{>0}$ so that  $Q(X_i) \in \{\emptyset, \{m\}\}$.
    \item[3)] For all $i < j$ such that $X_i$ and $X_j$ are non-empty, we have $Q(X_i) \le Q(X_j)$.
\end{enumerate}   
\end{definition}

\begin{example}
    Let $Q: \Z^{\oplus 2}_{\tilde{*}} \to \Z_{\ge 0}$ be a quadratic form defined by $Q(x,y) \defeq x^2 + y^2$. Then consider the sequence $(X_i)_{i=1}^{\infty}$ given by $X_i \defeq Q^{-1}(\{i\})$. Every strongly primitive vector appears in exactly one $X_i$, each set maps under $Q$ to the squared euclidean norm of the vectors in the set (and the norm is constant on each set), and the sets are arranged in increasing order of the norm. Therefore $(X_i)_{i=1}^{\infty}$ is  a successive minima sequence of $Q$. 
\end{example}

The following lemma gives us a way to construct a successive minima sequence that satisfies a desired property under $\minimum$ of subsequences. 

\begin{lemma}[\textbf{$\minimum$ Successive Minima Sequences}]
\label{lemma:succ_min}
Suppose $Q \in \mathcal{V}$ is a $\mathrm{GL}_2(\ZZ)$-reduced  form. Then there exists a sequence of vectors $(\vx_i)_{i=1}^{\infty}$ so that $(\{\vx_i\})_{i=1}^{\infty}$ is a successive minima sequence of $Q$  , and \[\vx_{i+1} \in \minimum(\Z^{\oplus 2}_{\tilde{*}} \setminus \{\vx_1, \ldots, \vx_i\}) \emph{ for } i \ge 0. \] When $i = 0$, $\vx_1 \in \minimum(\Z^{\oplus 2}_{\tilde{*}})$.  
\end{lemma}

\begin{proof}
We define the sequence inductively. By \autoref{lemma:min_alt_def}, there exists $\vx_1 \in \minimum(\Z^{\oplus 2}_{\tilde{*}})$ so that $Q(\vx_1) = \min Q(\Z^{\oplus 2}_{\tilde{*}})$, i.e. $Q(\vx_1) \le Q(\vx)$ for all $\vx \in \Z^{\oplus 2}_{\tilde{*}}$. Suppose we have constructed the sequence up to $\vx_1, \ldots, \vx_i$. Then by yet another application of \autoref{lemma:min_alt_def}, we can find $\vx_{i+1} \in \minimum(\Z^{\oplus 2}_{\tilde{*}} \setminus \{\vx_1, \ldots, \vx_i\})$ so that $Q(\vx_{i+1}) = \min Q(\Z^{\oplus 2}_{\tilde{*}} \setminus \{\vx_1, \ldots, \vx_i\})$.   

We prove that this is a successive minima sequence. At each step, we are choosing the vector that minimizes $Q$ on the vectors that have not been chosen yet. This implies that the $Q(x_i)$ are increasing. So it suffices to show that every vector in $\Z^{\oplus 2}_{\tilde{*}}$ appears in the sequence. 

Suppose some vector $\vy \in \Z^{\oplus 2}_{\tilde{*}}$ does not appear in the sequence. Since $Q$ is positive-definite, there are finitely many $\Z^{\oplus 2}_{\tilde{*}}$-solutions to the equation $Q(\vx) = m$ for any $m \in \Z_{>0}$. Thus there exists an index $i$ such that $Q(\vx_i) \le Q(\vy) < Q(\vx_{i+1})$. By construction, $Q(\vx_{i+1}) = \min Q(\Z^{\oplus 2}_{\tilde{*}} \setminus \{\vx_1, \ldots, \vx_i\})$ but $\vy \in \Z^{\oplus 2}_{\tilde{*}} \setminus \{\vx_1, \ldots, \vx_i\}$ and $Q(\vy) < Q(\vx_{i+1})$, a contradiction.

Thus the constructed sequence is a successive minima sequence.  
\end{proof}

Let $n \in \ZZ_{\ge 0}$. For the extended refinement algorithm in \autoref{sec:refinement_algo}, we need to iterate over all possibilities for the next $n$ minimal vectors. In other words, we want to find vectors $\vx_1, \ldots, \vx_n$ so that $\vx_i \in \minimum(X \setminus \{\vx_1, \ldots, \vx_{i-1}\})$ for all $i$. This motivates us to define the useful generalization $\minimum_n(X)$ as the set of all such possible choices (up to permutation). Formally, 

\medskip
\begin{definition}[\textbf{$\mathbf{n}^{\text{th}}$-order Minimal Subset $\minimum_n(X)$}]
\label{def:nth_order_minimal_subset}
    Suppose $X \subseteq \Z^{\oplus 2}$ and $n \in \ZZ_{\ge 1}$. Then we define the \textbf{$\mathbf{n}^{\emph{th}}$-order minimal subset}  $\minimum_n(X)$ \textbf{of $X$} by

    \[
  \minimum_n(X) \defeq \left\{ \mathcal{X} \subseteq \ZZ^{\oplus 2}_{\tilde{*}}\ \middle\vert \begin{array}{l}
    \#\mathcal{X} = n \text{ and for some ordering } \mathcal{X} \defeq \{\vx_1, \ldots, \vx_n\}, \text{ we have }  \\ 
              \quad \vx_{i+1} \in \minimum(X \setminus \{\vx_1, \ldots, \vx_{i}\})  \quad \text{ for each } \quad 0 \le i \le n-1
  \end{array}\right\}.
\]
\end{definition}

\begin{remark}[\textbf{$\minimum_1$ as $\minimum$}]
For $n = 1$ the notion of $\minimum_n(X)$ and $\minimum(X)$ coincide up to one level of unpacking (i.e., if $\minimum(X) = \{\vx_1, \ldots, \vx_{\# \minimum(X)}\}$ then $\minimum_1(X) = \{\{\vx_1\}, \ldots, \{\vx_{\# \minimum(X)}\}\}$).
\end{remark}

\begin{remark}[\textbf{Computing $\minimum_n$}]
\label{rem:computing_min_n}
    We can compute $\minimum_n(X)$ recursively by first computing $\minimum_{n-1}(X)$, then computing  $M_{n} \defeq \minimum_1(X \setminus \{\vx_1, \ldots, \vx_{n-1}\})$ (as in \autoref{rem:compute_min})  for each $\{\vx_1, \ldots, \vx_{n-1}\} \in \minimum_{n-1}(X)$, and setting $\vx_n \defeq \vx$ for each $\{\vx\} \in M_n$.
\end{remark}

The following corollary of \autoref{lemma:succ_min} follows directly from the definition of $\minimum_n$.
\begin{corollary}[\textbf{$\minimum_n$ Successive Minima Sequences}]
\label{corollary:succ_min_n}
    Suppose $(\vx_i)_{i=1}^{\infty}$ is a sequence satisfying the properties in \autoref{lemma:succ_min}. Then for all $m, n \in \Z_{\ge 0}$, we have \[\{\vx_{m+1}, \ldots, \vx_{m+n}\} \in \minimum_n(\Z^{\oplus 2}_{\tilde{*}} \setminus \{\vx_1, \ldots, \vx_m\}).\]
\end{corollary}

\section{The \texorpdfstring{$\kset$}{K}-set \texorpdfstring{$\kset(X; X_1, \ldots, X_n)$}{K(X, x1, ..., xn)}}

In the previous section we defined a successive minima sequence with respect to a quadratic form $Q$. In this section we flip the question to ask, ``Can we find all quadratic forms $Q \in \overline{\mathcal{V}}$ which have a given successive minima sequence?''. The $\kset$-set is the answer to this question.  

\medskip

Given a sequence $X_1, \ldots, X_n$, we define the $\kset$-set of this sequence as the set of all quadratic forms $Q$ so that $X_1, \ldots, X_n$ is a (truncated) successive minima sequence of $Q$. Formally,

\begin{definition}[\textbf{$\kset$-set} $\kset(X; X_1, \ldots, X_k)$]
\label{def:kset_set}
    Suppose $X \subseteq \Z^{\oplus 2}_{\tilde{*}}$, and $X_1, \ldots, X_k \subseteq X$ are finite pairwise-disjoint sets. We define the \textbf{$\kset$-set} $\kset(X; X_1, \ldots, X_k)$ by
    \[\kset(X; X_1, \ldots, X_k) \defeq \left\{ Q \in \overline{\mathcal{V}}\ \middle\vert \begin{array}{l}
    \emph{for each } 1 \le i \le k \emph{ we have } Q(X_i) \in \{\emptyset, \{m_i\}\}  \\
     \emph{where } m_i \defeq \min (Q(X \setminus (X_1 \union \ldots \union X_{i-1})))
  \end{array}\right\}.\]
\end{definition}

We show that the $\kset$-set is a polyhedral cone, and provide explicit inequalities required to algorithmically compute the $\kset$-set. First, we prove a special case.  

\begin{lemma}[\textbf{Singleton $\kset$-set Polyhedral Cone}]
\label{lemma:Kset_schiemann}
    Suppose $\vx_1, \ldots, \vx_k \in X \subseteq \Z^{\oplus 2}_{\tilde{*}}$ is a sequence of strongly primitive vectors in $X$. Then $\kset(X; \{\vx_1\}, \ldots, \{\vx_k\})$ is a polyhedral cone, explicitly given by 
    \[\kset(X; \{\vx_1\}, \ldots, \{\vx_k\}) = \left\{ Q \in \overline{\mathcal{V}}\ \middle\vert \begin{array}{l}
    \quad Q(\vx_1) \le \ldots \le Q(\vx_k) \le Q(\vx) \\ 
    \emph{for all } \vx \in \minimum(X \setminus \{\vx_1, \ldots, \vx_k\})
  \end{array}\right\}.\]
\end{lemma}

\begin{proof}
This is a rephrasing of \cite[\S5.1.1, Lemma 5.1.11, p39]{schiemann1997ternary}, replacing $\vx_i$ by $\{\vx_i\}$ for all $1 \le i \le k$. 
\end{proof}

We now use this to prove the general case. 

\begin{lemma}[\textbf{General $\kset$-set Polyhedral Cone}]
\label{lemma:Kset}
    Suppose $X_1, \ldots, X_k \subseteq X \subseteq \Z^{\oplus 2}_{\tilde{*}}$ are finite pairwise-disjoint subsets of $X$ with some fixed ordering $X_i \defeq \{\vx_{i,1}, \ldots, \vx_{i, \#X_i} \}$ for each  $1 \le i \le k$. Then $\kset(X; X_1, \ldots, X_k)$ is a polyhedral cone, explicitly given by
    \begin{equation}
    \label{eqn:kset}
        \begin{split}
             \kset(X; X_1, \ldots, X_k) = K'_{X; X_1, \ldots, X_k} \intersect L_{X; X_1, \ldots, X_k} \\ 
        \end{split}
    \end{equation}
    where 
    \begin{equation*}
        \begin{split}
            K'_{X; X_1, \ldots, X_k} &\defeq   \kset(X; \{\vx_{1,1}\}, \ldots, \{\vx_{1, \#X_1}\}, \ldots, \{\vx_{k,1}\}, \ldots, \{\vx_{k, \#X_k}\} ), \\
            L_{X; X_1, \ldots, X_k} &\defeq \{Q \in \overline{\mathcal{V}} \mid  Q(\vx_{i, 1}) = \cdots = Q(\vx_{i,\#X_i}) \emph{ for each } 1 \le i \le k\}, \\            
        \end{split}
    \end{equation*}
    and the sets $X_{i, j} \defeq \{\vx_{i, j}\}$ are ordered in increasing lexicographic order on the indexing pair $(i, j)$. 
\end{lemma}

\begin{proof}
We can rewrite the definition as 
\begin{equation*}
\begin{split}
 \kset(X; X_1, \ldots, X_k) =& \:  \left\{ Q \in \overline{\mathcal{V}}\ \middle\vert \begin{array}{l}
    Q(\vx_{i,j}) = \min(Q(X\setminus\{\vx_{1,1}, \ldots, \vx_{i, j-1}\}) \\
   \quad \text{ and } Q(\vx_{i,1}) = \cdots = Q(\vx_{i, \#X_k})
  \end{array}\right\}\\ 
  =& \: K'_{X; X_1, \ldots, X_k} \intersect L_{X; X_1, \ldots, X_k}. 
\end{split}
\end{equation*}
$K'_{X; X_1, \ldots, X_k}$ is a polyhedral cone (by \autoref{lemma:Kset_schiemann}), and $L_{X; X_1, \ldots, X_k}$ is a polyhedral cone (since it is an intersection of a subspace with the cone $\overline{\mathcal{V}}$). Therefore their intersection is also a polyhedral cone. 
\end{proof}
\begin{remark}
    In \autoref{eqn:kset}, it is not necessary to order the sets $X_{i, j} = \{\vx_{i, j}\}$ in increasing lexicographic order. They only need to be in increasing order of index $i$, with ties among the index $j$ broken arbitrarily. 
\end{remark}

\section{The Extended Refinement}

\subsection{Irrational and Normalized Linear Relations}
  
\begin{lemma}[\textbf{Irrational 3-term Linear Relations}]
\label{lemma:irrational}
  Suppose $Q_1, Q_2, Q_3$ are positive-definite integer-valued  binary quadratic forms satisfying 
  \[\alpha_1 \uptheta_{Q_1} + \alpha_2 \uptheta_{Q_2} + \alpha_3 \uptheta_{Q_3} = 0,\]
   for some $\alpha_1, \alpha_2, \alpha_3 \in \RR$, where one of the ratios $\frac{\alpha_i}{\alpha_j} \not\in \QQ$ for some $1 \le i, j \le 3$ with $\alpha_j \neq 0$. Then  $\uptheta_{Q_1} = \uptheta_{Q_2} = \uptheta_{Q_3}$.  
\end{lemma}
\begin{proof}
Without loss of generality, we can assume $\alpha_1, \alpha_2 \ge 0$ and $\alpha_3 < 0$. Then, letting $\alpha \defeq -\frac{\alpha_1}{\alpha_3}$ we can rewrite  \autoref{eqn:three_term_linear} as 
\begin{equation}
\label{eqn:alpha_1_eqn}
\alpha \uptheta_{Q_1} + (1 - \alpha) \uptheta_{Q_2} = \uptheta_{Q_3}.    
\end{equation}
If $\alpha \not\in \Q$, then we must have $\uptheta_{Q_1} = \uptheta_{Q_2}$, since if $c_1, c_2$ are the coefficients of $q^m$ in $\uptheta_{Q_1}$ and $\uptheta_{Q_2}$ respectively for some $m$, then $c_1 \neq c_2  \iff \alpha c_1 + (1 -\alpha) c_2 = y + \alpha (c_1 - c_2) \not\in \Q$. Plugging $\uptheta_{Q_1} = \uptheta_{Q_2}$ in the equation gives us $\uptheta_{Q_1} = \uptheta_{Q_2} = \uptheta_{Q_3}$, as required.      
\end{proof}
\begin{remark}[\textbf{Solving Irrational 3-term Linear Relations}]
    \autoref{lemma:irrational} shows that solving irrational 3-term linear relations is equivalent to solving the rational 2-term relation $\uptheta_{Q_1} = \uptheta_{Q_2}$. 
\end{remark}

\smallskip
From \autoref{lemma:irrational} we reduce to solving rational linear relations, which we now normalize.
\begin{definition}[\textbf{Normalized Rational 3-term Linear Relations}]
\label{def:normalized_rational}
     Given positive-definite binary quadratic forms $Q_1, Q_2, Q_3$ satisfying the relation \[\alpha_1\uptheta_{Q_1} + \alpha_2 \uptheta_{Q_2} + \alpha_3 \uptheta_{Q_3} = 0\] with $\alpha_1, \alpha_2, \alpha_3 \in \QQ$, we define its \textbf{normalized rational 3-term linear relation} by \[\beta_1 \uptheta_{Q_{\sigma(1)}} + \beta_2 \uptheta_{Q_{\sigma(2)}} = \uptheta_{Q_{\sigma(3)}},\] 
     with $\beta_1, \beta_2 \in \QQ_{\ge 0}$ and some permutation $\sigma$ of $\{1, 2, 3\}$.
\end{definition}
\begin{remark}[\textbf{Integral Form of Normalized Rational 3-term Linear Relation}]
\label{remark:integral_form}
We can write the normalized rational 3-term linear relation with $\beta_1 \defeq \frac{a}{c}$ and $\beta_2 \defeq \frac{b}{c}$ where $a, b, c \in \ZZ_{\ge 0}$, $c > 0$,  $\gcd(a, c) = \gcd(b, c) = 1$, and necessarily $c = a+b$ (by looking at the first term of the theta series). 
\end{remark}

\smallskip
\begin{definition}[\textbf{Solution Set $\mathcal{D}_{a, b}$ and Diagonal $\Delta$}]
  \label{def:solution_set}
  Given the integral form of the normalized rational 3-term linear relation 
  \begin{equation}
  \label{eqn:rational_lin}
      \frac{a}{a + b} \uptheta_{Q_1} + \frac{b}{a+b} \uptheta_{Q_2} = \uptheta_{Q_3},
  \end{equation}
  as described in   \autoref{remark:integral_form}, we define its \textbf{solution set}  $\mathcal{D}_{a, b}$ by  
  \[\mathcal{D}_{a, b} \defeq \{(Q_1, Q_2, Q_3) \in \mathcal{V} \times \mathcal{V} \times \mathcal{V} \mid \emph{ \autoref{eqn:rational_lin} holds}\}.\]
  We further define the \textbf{diagonal (solution set)} $\Delta \subseteq \mathcal{D}_{a, b}$ by 
  \[\Delta \defeq \{(Q_1, Q_2, Q_3) \in \mathcal{V} \times \mathcal{V} \times \mathcal{V} \mid Q_1 = Q_2 = Q_3\}.\]
\end{definition}

\subsection{The Key Lemma}

In this subsection we characterize non-negative integer solutions for three-term linear diophantine equations. This will be useful for iterating over minimal vectors in  \nameref{algo:extended_refinement_algorithm}.

\smallskip
We begin with the following definition.
\begin{definition}[\textbf{The Linset $\mathcal{L}_{a, b}$}]
\label{def:linset}
      Suppose $a, b \in  \Z_{\ge 0}$ with $(a, b) \neq (0, 0)$. Then we define the \textbf{linset} $\cL_{a, b}$  by \[\mathcal{L}_{a,b} \defeq \{L_1, L_2, L_3\}\]
      where 
      \[L_1 \defeq (1, 1, 1), \quad L_2 \defeq \frac{(a+b, 0, a)}{\gcd(a, b)}, \quad \text{ and } \quad L_3 \defeq \frac{(0, a+b, b)}{\gcd(a, b)}.\]
      We also denote the $j$-th coordinate of $L_i$ by $L_{i, j}$ for $1 \le i, j \le 3$.
\end{definition}
The following lemma highlights the significance of the linset.
\begin{lemma}[\textbf{The Key Lemma}]
\label{lemma:linear}
    Suppose $a, b \in \Z_{\ge 0}$ with $(a, b) \neq (0, 0)$. Then a non-negative tuple $(x_0, y_0, z_0) \in \Z_{\ge 0}^{\oplus 3}$ is a solution to the equation 
    \begin{equation}
        \label{eqn:linear_dio_eqn}
        \frac{a}{a+b} \: x + \frac{b}{a+b} \: y = z
    \end{equation}
    if and only if it is a $\ZZ_{\ge 0}$-linear combination of elements of $\mathcal{L}_{a, b}$ (i.e.
    \[(x_0, y_0, z_0) = c_1 L_1 + c_2 L_2 + c_3 L_3\]
    for some $c_1, c_2, c_3 \in \ZZ_{\ge 0}$).
\end{lemma}

\begin{proof} Without loss of generality we can assume that $\gcd(a, b) = 1$ since both \autoref{def:linset} and \autoref{eqn:linear_dio_eqn}  are invariant under the scaling $(a, b) \mapsto \frac{(a, b)}{\gcd(a, b)}$.

($\Longleftarrow$) Note that $(1, 1, 1), (a+b, 0 ,a), (0, a+b, b)$ are non-negative integer solutions to \autoref{eqn:linear_dio_eqn}. Thus any $\ZZ_{\ge 0}$-linear combination of these solutions is also a non-negative integral solution this equation, which suffices.

($\Longrightarrow$) Suppose $(x, y, z) = (x_0, y_0, z_0)$ is a solution to \autoref{eqn:linear_dio_eqn} with $x_0, y_0, z_0 \in \ZZ_{\ge 0}$. We claim that $\min\{x_0, y_0, z_0\} \in \{x_0, y_0\}$. Assume otherwise $\min\{x_0, y_0, z_0\} = z_0$. Then $(x_0 - z_0, y_0 - z_0, 0)$ is also a non-negative integer solution to \autoref{eqn:linear_dio_eqn}. Thus \[a(x_0 - z_0) + b(y_0 - z_0) = 0\] but since $a, b > 0$ we must have $x_0 = z_0$ and $y_0 = z_0$. Thus $\min\{x_0, y_0, z_0\} \in \{x_0, y_0\}$.

\textbf{Case 1}: Suppose $\min \{x_0, y_0, z_0\} = x_0$. Then $(0, y_0 - x_0, z_0 - x_0)$ is also a non-negative integer solution to \autoref{eqn:linear_dio_eqn}. This implies that \[b(y_0 - x_0) = (a+b)(z_0 - x_0),\] but since $\gcd(a+b, b) = \gcd(a, b) = 1$, by fundamental theorem of arithmetic we must have $y_0 - x_0 = (a+b)k$ and $z_0 - x_0 = bk$ for some $k \in \ZZ_{\ge 0}$. This shows that we have the desired $\ZZ_{\ge 0}$-linear combination
\[(x_0, y_0, z_0) = x_0 \cdot (1, 1, 1) + k \cdot (0, a+b, b).\] 

\textbf{Case 2}: Suppose $\min \{x_0, y_0, z_0\} = y_0$, then similarly we have 
\[(x_0, y_0, z_0) = y_0 \cdot (1, 1, 1) + k \cdot (a+b, 0, a)\] for some $k \in \ZZ_{\ge 0}$, which completes the proof. 
\end{proof}

\begin{remark}[\textbf{Linset Minimality}]
\label{rem:linset_minimality}
    When both $a \neq 0$ and $b \neq 0$, then the vectors $L_1, L_2, L_3$ in $\cL_{a, b}$ are the minimal set of generators for the $\ZZ_{\ge 0}$-span of $\cL_{a, b}$, given by \[\ZZ_{\ge 0} [\mathcal{L}_{a, b}] \defeq \{c_1 L_1 + c_2 L_2 + c_3 L_3 \mid c_1, c_2, c_3 \in \ZZ_{\ge 0}\}.\] When $a = 0$ or $b = 0$, the vectors $L_1, L_2, L_3$ are not a minimal set of generators for $\ZZ_{\ge 0} [\mathcal{L}_{a, b}] $, since $L_1 = L_2 + L_3$. A minimal set $\cL_{a, b}'$ of generators for $\ZZ_{\ge 0} [\mathcal{L}_{a, b}]$ is given by $\cL_{a, b}' \defeq \{L_2, L_3\}$. 
\end{remark}

\subsection{The Extended Refinement Algorithm}
\label{sec:refinement_algo}

In this subsection we state the extended refinement algorithm and prove its correctness for any given pair $a, b \in \ZZ_{\ge 0}$ with $(a, b) \neq (0, 0)$.

\medskip
We begin by defining the notion of a covering. 

\begin{definition}[\textbf{Covering}]
\label{def:covering}
     Suppose $S \subseteq \R^n$. Then we say $\mathcal{F} \defeq \{U \mid U \subseteq \R^n\}$ is a \textbf{covering of $S$} if  \[S \subseteq \bigcup_{U \in \mathcal{F}} U.\] 
\end{definition}

Next we define invariant properties which allow us to iteratively refine coverings. 

\begin{definition}[\textbf{Covering Parameter}] 
\label{def:covering_parameter}
We define a \textbf{covering parameter} $P$ as a tuple $P \defeq ((X_i)_{i=1}^k, (Y_i)_{i=1}^k, (Z_i)_{i=1}^k)$ where the $X_i, Y_i, Z_i \subseteq \ZZ^{\oplus 2}_{\tilde{*}}$ with $1 \le i \le k$ and $k \in \ZZ_{\ge 0}$. 
\end{definition}

\begin{definition}[\textbf{Admissibility of Covering Parameter}]
\label{def:covering_admissibility}
         Suppose $a, b \in \ZZ_{\ge 0}$ with $(a, b) \neq (0, 0)$ and  $T \subseteq \mathcal{V} \times \mathcal{V} \times \mathcal{V}$. Suppose also that $P \defeq ((X_i)_{i=1}^k, (Y_i)_{i=1}^k, (Z_i)_{i=1}^k)$ is a covering parameter with $k \in \ZZ_{\ge 0}$.  Then we say that $T$ \textbf{admits covering parameter $P$ (for the pair $(a, b)$)} if the following properties hold:
        \begin{enumerate}[label*=\arabic*.]
                \item[1)] $\vec{n} \defeq (\#X_i, \#Y_i, \#Z_i) \in \mathcal{L}_{a,b}$ for each $1 \le i \le  k$,
                \item[2)] $T \subseteq \kset(\Z^{\oplus 2}_{\tilde{*}}; X_1, \ldots, X_k) \times \kset(\Z^{\oplus 2}_{\tilde{*}}; Y_1, \ldots, Y_k) \times \kset(\Z^{\oplus 2}_{\tilde{*}}; Z_1, \ldots, Z_k)$, and 
                \item[3)] $T \subseteq \left\{ (Q_1, Q_2, Q_3) \in \overline{\mathcal{V}} \times \overline{\mathcal{V}} \times \overline{\mathcal{V}} \ \middle\vert \begin{array}{l}
                 \emph{for each } 1 \le i \le k, \emph{ there exists } m_i \in \ZZ_{\ge 0} \emph{ so that} \\ 
                \:\:\: \emph{we have } Q_1(X_i), Q_2(Y_i), Q_3(Z_i) \in \{\emptyset, \{m_i\}\}
                \end{array}\right\}$.
        \end{enumerate}
\end{definition}

\begin{definition}[\textbf{Covering Parameter Auxiliary Sets}]
\label{def:auxiliary_sets}
    Suppose $P \defeq ((X_i)_{i=1}^k, (Y_i)_{i=1}^k, (Z_i)_{i=1}^k)$ is a covering parameter with  $k \in \ZZ_{\ge 0}$. Further suppose $\vec{n} \defeq (n_1, n_2, n_3) \in \mathcal{L}_{a,b}$  and 
    \begin{equation*}
        \begin{split}
            \mathcal{X} &\defeq \mathcal{X}_{\vec{n}} \in \operatorname{MIN}_{n_1}\left(\mathbb{Z}_{\tilde{*}}^{\oplus 2} \setminus \left (X_1 \union \ldots \union X_k\right)\right), \\
            \mathcal{Y} &\defeq \mathcal{Y}_{\vec{n}} \in \operatorname{MIN}_{n_2}\left(\mathbb{Z}_{\tilde{*}}^{\oplus 2} \setminus \left(Y_1 \union \ldots \union Y_k\right)\right), \\
             \mathcal{Z} &\defeq \cZ_{\vec{n}} \in \operatorname{MIN}_{n_3}\left(\mathbb{Z}_{\tilde{*}}^{\oplus 2} \setminus \left(Z_1 \union \ldots \union Z_k\right)\right). \\
        \end{split}
    \end{equation*}
Then we define \textbf{auxiliary sets} (associated to the parameters $\cX, \cY, \cZ, P, \vec{n}$) by
\begin{equation*}
\begin{split}
\mathcal{K}_{\cX, \cY, \cZ} &\defeq \mathcal{K}_{\cX,\cY,\cZ; P, \vec{n}}  \defeq \kset(\Z^{\oplus 2}_{\tilde{*}}; X_1, \ldots, X_k, \cX) \times \kset(\Z^{\oplus 2}_{\tilde{*}}; Y_1, \ldots, Y_k, \cY) \times \kset(\Z^{\oplus 2}_{\tilde{*}}; Z_1, \ldots, Z_k, \cZ), \\
\mathcal{Q}_{\cX,\cY,\cZ} &\defeq \mathcal{Q}_{\cX,\cY,\cZ; P, \vec{n}}  \defeq \{(Q_1, Q_2, Q_3) \in \overline{\mathcal{V}} \times \overline{\mathcal{V}} \times \overline{\mathcal{V}}\mid Q_1(\mathcal{X}), Q_2(\mathcal{Y}), Q_3(\mathcal{Z}) \in \{\emptyset, \{m\}\} \emph{ for some } m \in \ZZ_{\ge 0}\}. \\
\end{split}
\end{equation*}
\end{definition}

\begin{definition}[\textbf{The Refinement Process}]
\label{def:refinement_procedure}
 Suppose $a, b \in \ZZ_{\ge 0}$ with $(a, b) \neq (0, 0)$,  and $T \subseteq \mathcal{V} \times \mathcal{V} \times \mathcal{V}$ admits covering parameter $P \defeq ((X_i)_{i=1}^k, (Y_i)_{i=1}^k, (Z_i)_{i=1}^k)$ with  $k \in \ZZ_{\ge 0}$. Then we define the \textbf{refinement $\mathcal{M}_{a, b}((T, P))$ of the pair $(T, P)$} as the set of pairs 
\begin{equation}
\label{eqn:Mab}
    \mathcal{M}_{a, b}((T, P))  \defeq\bigcup_{\substack{(n_1, n_2, n_3) \\ \in \mathcal{L}_{a, b}}} \:\: \bigcup_{\substack{ \mathcal{X} \in \operatorname{MIN}_{n_1}\left(\mathbb{Z}_{\tilde{*}}^{\oplus 2} \setminus\left(X_1\union \ldots\union X_k\right)\right) \\ \mathcal{Y} \in \operatorname{MIN}_{n_2}\left(\mathbb{Z}_{\tilde{*}}^{\oplus 2} \setminus\left(Y_1 \union \ldots \union Y_k\right)\right) \\ \mathcal{Z} \in \operatorname{MIN}_{n_3}\left(\mathbb{Z}_{\tilde{*}}^{\oplus 2} \setminus\left(Z_1\union \ldots\union Z_k\right)\right) }}\left\{(T_{\cX, \cY, \cZ}, P_{\cX, \cY, \cZ; \, T})\right\},
\end{equation}
where we define the \textbf{refined set $T_{\cX, \cY, \cZ}$} by \[T_{\cX, \cY, \cZ} \defeq T \intersect \mathcal{K}_{\cX,\cY,\cZ} \intersect \mathcal{Q}_{\cX,\cY,\cZ}\] 
and the \textbf{covering parameter $P_{\cX, \cY, \cZ; \, T}$ associated to $T_{\cX, \cY, \cZ}$} by \[P_{\cX, \cY, \cZ; \, T} \defeq ((X_i)_{i=1}^{k+1}, (Y_i)_{i=1}^{k+1}, (Z_i)_{i=1}^{k+1}),\] where $X_{k+1} \defeq \cX$, $Y_{k+1} \defeq \cY$, and $Z_{k+1} \defeq \cZ$. Finally, we define the \textbf{refinement of a set $\mathcal{T}$ of pairs  $(T, P)$} by
$$
\mathtt{refine}_{a, b}(\mathcal{T}):=\bigcup_{(T, P) \in \mathcal{T}} \mathcal{M}_{a, b}((T, P)).
$$ 
\end{definition}

\begin{remark}[\textbf{Covering Parameter Uniqueness}]
Suppose two associated covering parameters $P_{\cX, \cY, \cZ; \; T}$ and $P_{\cX', \cY', \cZ'; \; T}$ agree. Then the triples of sequences $((X_i)_{i=1}^{k+1}, (Y_i)_{i=1}^{k+1}, (Z_i)_{i=1}^{k+1})$ and \linebreak $((X'_i)_{i=1}^{k+1}, (Y'_i)_{i=1}^{k+1}, (Z'_i)_{i=1}^{k+1})$ must agree,  implying $\cX = \cX', \cY = \cY', \cZ = \cZ'$. Thus every covering parameter appears exactly once in $\mathcal{M}_{a, b}((T, P))$.       
\end{remark}

\begin{lemma}[\textbf{Covering Parameter of a Refined Set}]
    Suppose $T \subseteq \mathcal{V} \times \mathcal{V} \times \mathcal{V}$ admits covering parameter $((X_i)_{i=1}^k, (Y_i)_{i=1}^k, (Z_i)_{i=1}^k)$ with  $k \in \ZZ_{\ge 0}$. Then $T_{\cX, \cY, \cZ}$ admits the covering parameter $P_{\cX, \cY, \cZ; \, T}$ associated to $T_{\cX, \cY, \cZ}$.
\end{lemma}

\begin{proof}
    We will verify that properties 1) - 3) in \autoref{def:covering_admissibility} hold.

    \begin{enumerate}
         \item[1)] By assumption, $(\#X_i, \#Y_i, \#Z_i) \in \cL_{a, b}$ for $1 \le i \le k$. For $i = k+1$, recall that $X_i = \mathcal{X} \in \operatorname{MIN}_{n_1}\left(\mathbb{Z}_{\tilde{*}}^{\oplus 2} \setminus \left (X_1 \union \ldots \union X_k\right)\right)$, 
            $Y_i = \mathcal{Y}  \in \operatorname{MIN}_{n_2}\left(\mathbb{Z}_{\tilde{*}}^{\oplus 2} \setminus \left(Y_1 \union \ldots \union Y_k\right)\right)$, and
             $Z_i = \mathcal{Z} \in \operatorname{MIN}_{n_3}\left(\mathbb{Z}_{\tilde{*}}^{\oplus 2} \setminus \left(Z_1 \union \ldots \union Z_k\right)\right)$ where $(n_1, n_2, n_3) \in \cL_{a,b}$. So $(\#X_i, \#Y_i, \#Z_i) = (n_1, n_2, n_3) \in \cL_{a,b}$.

         \item[2)] Recall that $T_{\cX, \cY, \cZ} = T \intersect \mathcal{K}_{\cX,\cY,\cZ} \intersect \mathcal{Q}_{\cX,\cY,\cZ}$ and therefore \[T \subseteq \mathcal{K}_{\cX, \cY, \cZ} = \kset(\Z^{\oplus 2}_{\tilde{*}}; X_1, \ldots, X_k, \cX) \times \kset(\Z^{\oplus 2}_{\tilde{*}}; Y_1, \ldots, Y_k, \cY) \times \kset(\Z^{\oplus 2}_{\tilde{*}}; Z_1, \ldots, Z_k, \cZ).\]

         \item[3)] We have 
         \[T_{\cX, \cY, \cZ} \subseteq T \subseteq \left\{ (Q_1, Q_2, Q_3) \in \overline{\mathcal{V}} \times \overline{\mathcal{V}} \times \overline{\mathcal{V}} \ \middle\vert \begin{array}{l}
                 \text{for each } 1 \le i \le k, \text{ there exists } m_i \in \ZZ_{\ge 0} \text{ so that} \\ 
                \:\:\: \text{we have } Q_1(X_i), Q_2(Y_i), Q_3(Z_i) \in \{\emptyset, \{m_i\}\}
                \end{array}\right\}.\]

         For $i = k+1$, since $T_{\cX, \cY, \cZ} = T \intersect \mathcal{K}_{\cX,\cY,\cZ} \intersect \mathcal{Q}_{\cX,\cY,\cZ}$ we have \[T \subseteq \mathcal{Q}_{\cX,\cY,\cZ} =  \{(Q_1, Q_2, Q_3) \in \overline{\mathcal{V}} \times \overline{\mathcal{V}} \times \overline{\mathcal{V}}\mid Q_1(\mathcal{X}), Q_2(\mathcal{Y}), Q_3(\mathcal{Z}) \in \{\emptyset, \{m_{k+1}\}\} \text{ for some } m_{k+1} \in \R_{>0}\},\]
         as desired.
     \end{enumerate} 
\end{proof}

\begin{definition}[\textbf{The Refinement Sequence}]
\label{def:refinement_sequence}
    Suppose $a, b \in \ZZ_{\ge 0}$ with $(a, b) \neq (0, 0)$. Then we define the \textbf{refinement sequence} $(\mathcal{T}_i)_{i=0}^{\infty}$ (associated to $(a, b)$) by 
    \[\mathcal{T}_i = \begin{cases}
        \left\{(T_0, P_0)\right\} &  \text{if } i = 0,\\
        \mathtt{refine}_{a, b}(\mathcal{T}_{i-1}) & \text{if } i \ge 1,\\
    \end{cases}\]
    where  $T_0 \defeq \mathcal{V} \times {\mathcal{V}} \times {\mathcal{V}}$, $P_0 \defeq ((), (), ())$, and $\mathtt{refine}_{a, b}(\mathcal{T}_{i-1})$ is as defined in \autoref{def:refinement_procedure}.
\end{definition}

Next we define the Extended Refinement Algorithm.

\medskip

\begin{algorithm}[H]
\label{algo:extended_refinement_algorithm}
 \caption{\textbf{Extended Refinement Algorithm}}

  {\textbf{input}}: $(a, b), \quad \mathtt{STOP\_SET}, \quad \mathtt{MAX\_ITERATIONS}$\;

\medskip 

 \SetKwProg{Fn}{function}{}{}
 \Fn{$\mathtt{{removeStopSetSubsets}}(\mathcal{S}, \: \mathtt{STOP\_SET})$}{
     {\textbf{initialize}} $\mathcal{S}' \defeq \{\}$\;
     \For{$(T, P)$ $\mathrm{\mathbf{in}}$ $\mathcal{S}$}{
    \If{$T \not\subseteq \mathtt{STOP\_SET}$}{
        {\textbf{update}} $\mathcal{S}'  = \mathcal{S}' \union \{(T, P)\}$;
    }
    }
    {\textbf{return}} $\mathcal{S}'$\;
 } 

 \medskip
 {\textbf{initialize}} $ \mathcal{S}_0 \defeq \{(T_0, P_0)\}$\;
 {\textbf{initialize}} $\mathtt{i} \defeq 0$\;
 \While{$\mathcal{S}_i \emph{ is not } \{\} \emph{ and } \mathtt{i} < \mathtt{MAX\_ITERATIONS}$}{
    {\textbf{initialize}} $ \mathcal{S}_{i+1} \defeq \mathtt{removeStopSetSubsets}(\mathcal{S}_i, \: \mathtt{STOP\_SET})$\;
    {\textbf{update}} $ \mathcal{S}_{i+1} = \mathtt{refine}_{a, b}(\mathcal{S}_{i+1})$\;
    {\textbf{update}} $\mathtt{i} = \mathtt{i} + 1$\;
    
 }

 \medskip

 {\textbf{output}}: $\mathcal{S}_i$\;

\end{algorithm}

\medskip
\begin{remark}
    The sets $\mathcal{S}_i$ in \nameref{algo:extended_refinement_algorithm} differ slightly from the sets $\mathcal{T}_i$ in the refinement sequence, since $\mathcal{S}_i$ does not contain any pairs $(T, P)$ satisfying $T \subseteq \mathtt{STOP\_SET}$. 
\end{remark}

We now formally prove that the extended refinement algorithm gives a set of coverings of the desired solution set $\mathcal{D}_{a, b}$. First we prove some technical lemmas.

\begin{lemma}[\textbf{Characterization of $\mathcal{D}_{a, b}$ by Covering Parameters}]
\label{lemma:succ_min_extended}
    Suppose $(Q_1, Q_2, Q_3) \in \mathcal{V} \times \mathcal{V} \times \mathcal{V}$. Then $(Q_1, Q_2, Q_3) \in \mathcal{D}_{a, b}$ if and only if for each $i \in \{1, 2, 3\}$ there exists a successive minima sequence $S_i \defeq (S_{ij})_{j=1}^{\infty}$ of $Q_i$ so that for all $j \ge 1$, the following properties hold: 
    \begin{enumerate}
        \item[1)] $(\#S_{1j}, \#S_{2j}, \#S_{3j}) \in \mathcal{L}_{a,b}$,
        \item[2)] $Q_i(S_{ij}) \in \{\emptyset,  \{m_j\}\}$ for some $m_j \in \ZZ_{> 0}$, and 
        \item[3)] $S_{ij} \in \minimum_{\#S_{ij}}(\Z^{\oplus 2}_{\tilde{*}} \setminus (S_{i1} \union \cdots \union S_{i(j-1)}))$.
    \end{enumerate}
\end{lemma}

\begin{proof}
($\Longrightarrow$) By \autoref{lemma:succ_min}, for each $i \in \{1, 2, 3\}$ there exists a sequence $S'_i \defeq (\vec{s}_{ij})_{j=1}^{\infty}$ of vectors in $\ZZ^{\oplus 2}_{\tilde{*}}$ such that $(\{\vec{s}_{ij}\})_{j=1}^{\infty}$ is a successive minima sequence satisfying  
\[\vec{s}_{i(j+1)} \in \minimum(\Z^{\oplus 2}_{\tilde{*}} \setminus \{\vec{s}_{i1}, \ldots, \vec{s}_{ij}\}) \text{ for all } j \ge 1. \]

Fix $m \in \Z_{>0}$. The elements of the $\Z^{\oplus 2}_{\tilde{*}}$-preimage of $m$ under $Q_i$ can be viewed as a contiguous subsequence of $S'_i$ with $r^{\tilde{*}}_{Q_i}(m)$ elements. Therefore for each $i \in \{1, 2, 3\}$, there exist indices $k_i \defeq k_{i, m}$ and $\ell_i \defeq \ell_{i, m}$ such that 
\[Q_i\left(\{\vs_{ij}\}_{j=k_i}^{\ell_i}\right) \in \{\emptyset, \{m\}\} \quad \text{ and } \quad r^{\tilde{*}}_{Q_i}(m) = \ell_i - k_i + 1.\]

By examining the coefficient of $q^m$ in \autoref{eqn:rational_lin}, we see that the triple $(r^{\tilde{*}}_{Q_1}(m), r^{\tilde{*}}_{Q_2}(m), r^{\tilde{*}}_{Q_3}(m))$ satisfies \autoref{eqn:linear_dio_eqn}. Therefore by \autoref{lemma:linear}, it must be a $\ZZ_{\ge 0}$-linear combination of the elements of $\mathcal{L}_{a, b} = \{L_1, L_2, L_3\}$,  i.e., there exist non-negative integers $c_1, c_2, c_3 \in \Z_{\ge 0}$ so that
\[(r^{\tilde{*}}_{Q_1}(m), r^{\tilde{*}}_{Q_2}(m), r^{\tilde{*}}_{Q_3}(m)) = c_1 L_1 + c_2 L_2 + c_3 L_3 =  c_1 \cdot (1, 1, 1) + c_2 \cdot (a+b, 0, a) + c_3 \cdot (0, a+b, b).\]

 Next we partition the subsequence $(\vs_{ij})_{j=k_i}^{\ell_i}$ of $S_i'$ into a sequence $(X_{m, n; \; i})_{n=1}^{c_1 + c_2 + c_3}$ of $c_1 + c_2 + c_3$ (possibly empty) sets as follows: start from the beginning of the subsequence $(\vs_{ij})_{j=k_i}^{\ell_i}$, and create a sequence of $c_1$ sets 
\begin{equation}
\label{eqn:1_to_c1}
 \left(X_{m, n; \; i} \defeq \{\vec{s_{ij}}\}_{j=k_i + (n - 1) L_{1, i}}^{k_i + n L_{1, i} - 1}\right)_{n=1}^{c_1}   
\end{equation}
 each having cardinality $L_{1, i}$, then create a sequence of $c_2$ sets 
\begin{equation}
\label{eqn:c1plus1_to_c1plusc2}
    \left(X_{m, n; \; i} \defeq \{\vec{s_{ij}}\}_{j=k_i + c_1L_{1, i} + (n - 1) L_{2, i}}^{k_i + c_1 L_{1, i} +n L_{2, i} - 1}\right)_{n=c_1 + 1}^{c_1 + c_2}
\end{equation}
 each having cardinality $L_{2, i}$, and finally create a sequence of $c_3$ sets 
\begin{equation}
\label{eqn:c1plusc2plus1_to_c1plusc2plusc3}
    \left(X_{m, n; \; i} \defeq \{\vec{s_{ij}}\}_{j=k_i + c_1L_{1, i} +  c_2L_{2, i} + (n-1)L_{3, i}}^{k_i + c_1 L_{1, i} +c_2 L_{2, i} + n L_{3, i} - 1}\right)_{n=c_1 + c_2 + 1}^{c_1 + c_2 + c_3}
\end{equation}
 each having cardinality $L_{3, i}$. (Recall from \autoref{def:linset} that we denote the $i'$-th coordinate of $L_{j'}$ by $L_{j', i'}$. For a more explicit definition of $X_{m, n; \; i}$, see \autoref{subsec:explicit_indexing}.)  

\smallskip
Finally, we define the sequence $S_i = (S_{ij})_{j=1}^{\infty}$ by arranging the sets $X_{m, n; \; i}$ by the indexing pair $(m, n)$, in increasing lexicographic order. (I.e. $S_{ij} \defeq X_{m, n; \; i}$ where $X_{m, n; \; i}$ is the $j$-th set under this ordering.) 

\smallskip
The sequences $S_i$ satisfy property 1) since by \autoref{eqn:1_to_c1}, \autoref{eqn:c1plus1_to_c1plusc2}, and \autoref{eqn:c1plusc2plus1_to_c1plusc2plusc3}, we have

\[(\#X_{m, n;\; 1}, \#X_{m, n;\; 2}, \#X_{m, n;\; 3}) = \begin{cases}
    (L_{1, 1}, L_{1, 2}, L_{1, 3}) = L_1 & \text{ if } 1 \le n \le c_1, \\
    (L_{2, 1}, L_{2, 2}, L_{2, 3}) = L_2 & \text{ if } c_1 + 1 \le n \le c_1 + c_2,  \\
    (L_{3, 1}, L_{3, 2}, L_{3, 3}) = L_3 & \text{ if } c_1 + c_2 + 1 \le n \le c_1 + c_2 + c_3. \\ 
\end{cases}\]

The sequences $S_i$ also satisfy property 2) since $Q_i (X_{m, n; \; i}) =  \{m\}$ (if non-empty) or $\emptyset$. By \autoref{corollary:succ_min_n} the sequences $S_i$ satisfy property 3), as desired. 

\smallskip
($\Longleftarrow$) Suppose the sequences $S_i$ satisfying properties 1) - 3) exist. Fix $m \in \ZZ_{> 0}$. Then there exist indices $k_m, \ell_m$ so that \[Q_i(S_{ij}) \in \{\emptyset, \{m\}\} \Longleftrightarrow k_m \le j \le \ell_m,\]
since $S_i$ is a successive minima sequence and by property 2). Furthermore, since the sequence $S_i$ is a successive minima sequence, every vector $\vec{x} \in \ZZ^{\oplus 2}_{\tilde{*}}$ that satisfies $Q_i(\vec{x}) = m$ must appear in exactly one $S_{ij}$ where $k_m \le j \le \ell_m$, showing that \[\sum\limits_{j = k_m}^{\ell_m} \# S_{ij} = r_{Q_i}^{\tilde{*}}(m).\]   
Therefore
\[(r_{Q_1}^{\tilde{*}}(m), r_{Q_2}^{\tilde{*}}(m), r_{Q_3}^{\tilde{*}}(m)) = \sum\limits_{j=k_m}^{\ell_m} (\# S_{1j}, \#S_{2j}, \#S_{3j}),\]
and by property 1) we have $(\# S_{1j}, \#S_{2j}, \#S_{3j}) \in \mathcal{L}_{a, b}$. Hence the sum is is a $\ZZ_{\ge 0}$-linear combination of elements of $\mathcal{L}_{a, b}$, which by \autoref{lemma:linear} satisfies the equation
\[\frac{a}{a+b} r_{Q_1}^{\tilde{*}}(m) + \frac{b}{a+b} r_{Q_2}^{\tilde{*}}(m) =  r_{Q_3}^{\tilde{*}}(m),\]
    proving the lemma.  
\end{proof}

\begin{lemma}[\textbf{Refinement Sequence Covering Parameters}]
\label{lemma:ref_cov_property}
    Suppose $(Q_1, Q_2, Q_3) \in \mathcal{D}_{a, b}$, and for each $i \in \{1, 2, 3\}$ let $S_i$  be a successive minima sequence of $Q_i$ as defined in \autoref{lemma:succ_min_extended}. Then for each $k \in \ZZ_{\ge 0}$, there exists some $(T, P) \in \mathcal{T}_k$  where  $P = ((S_{1j})_{j=1}^{k}, (S_{2j})_{j=1}^{k}, (S_{3j})_{j=1}^{k})$. 
\end{lemma}

\begin{proof}
    We proceed by induction on $k$. For $k = 0$, we have $((S_{1j})_{j=1}^{k}, (S_{2j})_{j=1}^{k}, (S_{3j})_{j=1}^{k}) = ((), (), ())$ and since $\mathcal{T}_0 = \{(\mathcal{V} \times \mathcal{V} \times \mathcal{V}, ((), (), ()))\}$, the base case holds.

    Next we establish the lemma for $k = k'+1$ assuming it holds for $k = k'$, that is, we know there exists some $(T, P) \in \mathcal{T}_{k'}$ so that $P = ((S_{1j})_{j=1}^{k'}, (S_{2j})_{j=1}^{k'}, (S_{3j})_{j=1}^{k'})$. Note that by property 1) of \autoref{lemma:succ_min_extended}, we have $(n_1, n_2, n_3) \defeq (\#S_{1(k'+1)}, \#S_{1(k'+1)}, \#S_{1(k'+1)}) \in \mathcal{L}_{a, b}$, and by property 3) of \autoref{lemma:succ_min_extended} we have

    \[S_{1(k'+1)}  \in \minimum_{n_1}(\Z^{\oplus 2}_{\tilde{*}} \setminus (S_{11} \union \cdots \union S_{1k'})),\]
    \[S_{2(k'+1)} \in \minimum_{n_2}(\Z^{\oplus 2}_{\tilde{*}} \setminus (S_{21} \union \cdots \union S_{2k'})),\]
    \[S_{3(k'+1)} \in \minimum_{n_3}(\Z^{\oplus 2}_{\tilde{*}} \setminus (S_{31} \union \cdots \union S_{3k'})).\] 
    Thus in \autoref{eqn:Mab} of the \autoref{def:refinement_procedure}, we can let \[\mathcal{X} \defeq S_{1(k'+1)}, \quad \mathcal{Y} \defeq S_{2(k'+1)}, \quad \mathcal{Z} \defeq S_{3(k'+1)},\] which gives the corresponding refinement $T_{\cX, \cY, \cZ}$ with associated covering parameter  \[P = P_{\cX, \cY, \cZ; \; T} = ((S_{1j})_{j=1}^{k'+1}, (S_{2j})_{j=1}^{k'+1}, (S_{3j})_{j=1}^{k'+1}),\]
    as desired.
\end{proof}

\medskip
\begin{theorem}[\textbf{Extended Refinement Algorithm Correctness}]
    Suppose $a, b \in \ZZ_{\ge 0}$ with $(a, b) \neq (0, 0)$, and $(\mathcal{T}_k)_{k=0}^{\infty}$ is the refinement sequence associated to $(a, b)$, and let \[\mathcal{U}_k \defeq \bigcup\limits_{(T, P) \in \mathcal{T}_k} T\] denote the union of the refined sets in each $\mathcal{T}_k$. Then we have \[\mathcal{U}_0 \supseteq \mathcal{U}_1 \supseteq \ldots  \supseteq \mathcal{D}_{a, b}.\]
    \end{theorem}

\begin{proof} \textbf{Part 1)} We first show that $\mathcal{U}_{k+1} \subseteq \mathcal{U}_k$ by showing that for any $(T, P) \in \cT_{k+1}$, there exists $(T', P') \in \cT_k$ with $T \subseteq T'$. Suppose $(T, P) \in \cT_{k+1}$ with $P \defeq ((X_j)_{j=1}^{k+1}, (Y_j)_{j=1}^{k+1}, (Z_j)_{j=1}^{k+1})$. Then there must exist $(T', P') \in \cT_k$ with $P' \defeq ((X_j)_{j=1}^{k}, (Y_j)_{j=1}^{k}, (Z_j)_{j=1}^{k})$. But by \autoref{def:refinement_procedure} this gives us \[T = T' \intersect \cK_{X_{k+1}, Y_{k+1}, Z_{k+1}} \intersect \cQ_{X_{k+1}, Y_{k+1}, Z_{k+1}}\] and so $T \subseteq T'$, as desired.

\textbf{Part 2)} Now we show that $\mathcal{U}_k \supseteq \mathcal{D}_{a, b}$. Let $(Q_1, Q_2, Q_3) \in \mathcal{D}_{a, b}$, and for each $i \in \{1, 2, 3\}$ let $S_i$  be a successive minima sequence of $Q_i$ as defined in \autoref{lemma:succ_min_extended}. By \autoref{lemma:ref_cov_property}, for each $k \in \ZZ_{\ge 0}$ there exists $(T_k, P_k) \in \cT_k$ with $P_k = ((S_{1j})_{j=1}^k, (S_{2j})_{j=1}^k, (S_{3j})_{j=1}^k)$. We now use induction on $k$ to show that $T_k$ contains $(Q_1, Q_2, Q_3)$. 

For $k = 0$ we have $\cT_0 = \{(T_0, P_0)\}$. Since $(Q_1, Q_2, Q_3) \in \mathcal{D}_{a, b} \subseteq T_0 = \mathcal{V} \times \mathcal{V} \times \mathcal{V}$, the base case holds. 

Now we show the theorem is true for $k = k'+1$  assuming that it is true for $k = k'$, i.e., we know $(Q_1, Q_2, Q_3) \in T_{k'}$ for $(T_{k'}, P_{k'}) \in \mathcal{T}_{k'}$ where $P_{k'} = ((S_{1j})_{j=1}^{k'}, (S_{2j})_{j=1}^{k'}, (S_{3j})_{j=1}^{k'})$. Now consider $(T_{k'+1}, P_{k'+1})$ where $P_{k'+1} = ((S_{1j})_{j=1}^{k'+1}, (S_{2j})_{j=1}^{k'+1}, (S_{3j})_{j=1}^{k'+1})$. By \autoref{def:refinement_procedure} we have
\[T_{k'+1} = \left(T_{k'} \intersect \cK_{S_{1(k'+1)}, S_{2(k'+1)}, S_{3(k'+1)}} \intersect \cQ_{S_{1(k'+1)}, S_{2(k'+1)}, S_{3(k'+1)}}\right)\]
where 
\[
\cK_{S_{1(k'+1)}, S_{2(k'+1)}, S_{3(k'+1)}} \defeq \kset(\Z^{\oplus 2}_{\tilde{*}}; S_{11}, \ldots, S_{1(k'+1)}) \times \kset(\Z^{\oplus 2}_{\tilde{*}}; S_{21}, \ldots, S_{2(k'+1)}) \times \kset(\Z^{\oplus 2}_{\tilde{*}}; S_{31}, \ldots, S_{3(k'+1)}),\]
\[\mathcal{Q}_{S_{1(k'+1)}, S_{2(k'+1)}, S_{3(k'+1)}}  \defeq \left\{(Q_1', Q_2', Q_3') \in \overline{\mathcal{V}} \times \overline{\mathcal{V}} \times \overline{\mathcal{V}} \ \middle\vert \begin{array}{l}
                 \qquad\quad \text{there exists } m \in \ZZ_{\ge 0} \text{ so that } \\
                Q_i'(S_{i(k'+1)}) \in \{\emptyset, \{m\}\} \text{ for each } i \in \{1, 2, 3\}\\ 
                \end{array}\right\}.\]

By definition of $\kset$-sets, we have $Q_1 \in \kset(\Z^{\oplus 2}_{\tilde{*}}; S_{11}, \ldots, S_{1(k'+1)})$, $Q_2 \in \kset(\Z^{\oplus 2}_{\tilde{*}}; S_{21}, \ldots, S_{2(k'+1)})$, $Q_3 \in \kset(\Z^{\oplus 2}_{\tilde{*}}; S_{31}, \ldots, S_{3(k'+1)})$, and thus $(Q_1, Q_2, Q_3) \in \cK_{S_{1(k'+1)}, S_{2(k'+1)}, S_{3(k'+1)}}$. By property 2) of \autoref{lemma:succ_min_extended}, we have $Q_1(S_{1(k'+1)})$, $Q_2(S_{2(k'+1)})$, $Q_3(S_{3(k'+1)}) \in \{\emptyset, \{m\}\}$ for some $m \in \ZZ_{\ge 0}$, and thus $(Q_1, Q_2, Q_3) \in \mathcal{Q}_{S_{1(k'+1)}, S_{2(k'+1)}, S_{3(k'+1)}}$. By the induction hypothesis, we have $(Q_1, Q_2, Q_3) \in T_{k'}$, which implies \[(Q_1, Q_2, Q_3) \in \left(T_{k'} \intersect \cK_{S_{1(k'+1)}, S_{2(k'+1)}, S_{3(k'+1)}} \intersect \cQ_{S_{1(k'+1)}, S_{2(k'+1)}, S_{3(k'+1)}}\right) = T_{k'+1}.\]  
We finally note that  $(Q_1, Q_2, Q_3) \in T_i \subseteq \mathcal{U}_i$, thus $\mathcal{U}_i \supseteq \mathcal{D}_{a, b}$ as desired.
\end{proof}

\subsection{Algorithmic Considerations} 
\label{sec:implementation}
In this subsection we describe key implementation details and algorithmic aspects for this project. 

\subsubsection{Implementation}
We implemented \nameref{algo:extended_refinement_algorithm} in \texttt{SageMath} (\cite{centrallyAntisymmetricPolyCone}, \cite{sagemath}). \texttt{SageMath} has a polyhedron interface to cdd (and several other packages), however we chose to build a custom polyhedral cone library for two key reasons: First, the standard polyhedra packages in \texttt{SageMath} do not accept open halfspace conditions, but the $\text{GL}_2(\ZZ)$-reduction domain $\mathcal{V}$ requires strict inequalities in its definition. Second, without caching the edge computations from each stage of the refinement, the default class becomes inefficient and does not scale well as the dimension of the problem increases (e.g. ternary quadratic forms as in Schiemann's work or $n$-term linear relations). Therefore we chose to implement it with these additional use cases in mind. For reference, \cite[\S4, p27-33]{rydell2020three} contains a treatment of the theory of polyhedral cones required for our implementation. 

\subsubsection{Computing Refinement Objects}
For clarity, we summarize how to compute objects from the refinement algorithmically: 

\medskip
\begin{enumerate}
    \item[1)] \textbf{$\mathbf{\cK_{\cX, \cY, \cZ}}$}: This is a cross-product of three $\kset$-sets. We know each $\kset$-set is a polyhedral cone that can be explicitly computed using \autoref{lemma:Kset}. Then $\cK_{\cX, \cY, \cZ}$ is  the direct sum of these polyhedral cones.  
    \item[2)] \textbf{$\cQ_{\cX, \cY, \cZ}$}: There are three possibilites on $(\# \cX, \# \cY, \# \cZ)$, each requiring that $Q_1(\cX), Q_2(\cY), Q_3(\cZ) \in \{\emptyset, \{m\}\}$ for some $m \in \RR_{> 0}$:
    \subitem i) If $(\# \cX, \# \cY, \# \cZ) = (1, 1, 1)$, then  $Q_1(\cX) = Q_2(\cY) = Q_3(\cZ)$. 
    \subitem ii) If $(\# \cX, \# \cY, \# \cZ) = (a+b, 0, a)$, then  $Q_1(\cX) = Q_3(\cZ)$ and $Q_2$ is arbitrary.
    \subitem iii) If $(\# \cX, \# \cY, \# \cZ) = (0, a+b, b)$, then $Q_2(\cY) = Q_3(\cZ)$ and $Q_1$ is arbitrary.
    
    In each case, the set of forms $(Q_1 ,Q_2, Q_3)$ satisfying these equations is a polyhedral cone. 

    \item[3)] \textbf{$\mathbf{\minimum}_n(X)$}:  See \autoref{rem:computing_min_n} for details.
\end{enumerate}
See \nameref{sec:appendix} for examples of explicit computations of different refinement objects. 

\subsubsection{Termination Conditions}

In Schiemann's original refinement algorithm (\cite{schiemann1997ternary}), termination occurs when a polyhedral cone is contained in the diagonal $\Delta' \subset \R^6 \times \R^6$. This is because in Schiemann's case, it was widely believed that the only reduced solutions $(Q_1, Q_2)$ of $\uptheta_{Q_1} = \uptheta_{Q_2}$ should satisfy $Q_1 = Q_2$.

For our three-term linear relations, there exist non-trivial solutions \underline{not} contained in the diagonal $\Delta \subset \R^3 \times \R^3 \times \R^3$. Therefore, for practical purposes, it is important to impose an additional terminating condition, which we include as $\mathtt{STOP\_SET}$ and $\mathtt{MAX\_ITERATIONS}$ in \nameref{algo:extended_refinement_algorithm}. 

\vspace{-10pt}
\section{Results}

In this section we present results obtained from implementing \nameref{algo:extended_refinement_algorithm}, proving \autoref{thm:theorem_1_0}.  

\medskip
Suppose $a, b \in \ZZ_{\ge 0}$ with $(a, b) \neq 0$ and $\gcd(a, b) = 1$. Then recall the normalized rational $3$-term linear relation  \[\frac{a}{a+b} \uptheta_{Q_1} + \frac{b}{a+b} \uptheta_{Q_2} = \uptheta_{Q_3}.\]  
We now give our results, grouped by the value of $a+b$.

\subsection{\texorpdfstring{$a + b = 1$}{a + b = 1}}
In this case $(a, b) = (1, 0)$ or $(0, 1)$, and by symmetry it suffices to consider only $(a, b) = (1, 0)$. Then the normalized rational 3-term linear relation is given by 
\begin{equation}
\label{eqn:bqf_theta_uniqueness}
  \uptheta_{Q_1} = \uptheta_{Q_3}.
\end{equation}
 We apply \nameref{algo:extended_refinement_algorithm} with $\mathtt{STOP\_SET} \defeq \Delta'$, where \[\Delta' \defeq \{(Q_1, Q_2, Q_3) \in \mathcal{V} \times \mathcal{V} \times \mathcal{V} \mid Q_1 = Q_3\}\] and $\mathtt{MAX\_ITERATIONS} \defeq 13$. Note that by \autoref{rem:linset_minimality}, we can use $\mathcal{L}_{1,0}'$ instead of $\mathcal{L}_{1, 0}$. 

 The algorithm terminates after $\mathtt{MAX\_ITERATIONS}$ iterations. By analyzing the pairs $(T, P)$ in $\mathcal{S}_{13}$ with $T \not\subseteq \Delta'$, we write $P = ((X_i)_{i=1}^{13}, (Y_i)_{i=1}^{13}, (Z_i)_{i=1}^{13})$, and notice that each such $P$ has at least one index $i'$ (with $1 \le i' \le 13$) where $\#Y_{i'} > 0$. However, since the coefficient of $\uptheta_{Q_2}$ in \autoref{eqn:bqf_theta_uniqueness} is $0$, choosing the set $Y_{i'}$ of strongly primitive vectors of $Q_2$ does not give us any additional information about the linear relation, because every such refinement can also be realized (with the same projection $\pi_{1, 3}: (Q_1, Q_2, Q_3) \mapsto (Q_1, Q_3)$) by removing the triple $(X_{i}', Y_{i}', Z_{i}')$ from $P$. By applying this logic repeatedly, we can ignore all such triples $(T, P)$ and assume $\#Y_i = 0$ for all $1 \le i \le 13$. Thus all the solutions are in the  $\mathtt{STOP\_SET} = \Delta'$, showing $Q_1 \sim_\ZZ Q_3$, as desired. 

\begin{corollary}
    Suppose $Q_1, Q_2, Q_3$ are positive-definite integer-valued  binary quadratic forms satisfying 
  \[\alpha_1 \uptheta_{Q_1} + \alpha_2 \uptheta_{Q_2} + \alpha_3 \uptheta_{Q_3} = 0,\]
   for some $\alpha_1, \alpha_2, \alpha_3 \in \RR$, where one of the ratios $\frac{\alpha_i}{\alpha_j} \not\in \QQ$ for some $1 \le i, j \le 3$ with $\alpha_j \neq 0$.  Then  ${Q_1} \sim_\ZZ {Q_2} \sim_\ZZ {Q_3}$.  
\end{corollary}    

\begin{proof}
    By \autoref{lemma:irrational}, it suffices to solve $\uptheta_{Q_1} = \uptheta_{Q_2} = \uptheta_{Q_3}$. However, we know from this subsection that any pairwise equality $\uptheta_{Q_i} = \uptheta_{Q_j}$ for $i \neq j$ implies  $Q_i \sim_\ZZ Q_j$, giving $Q_1 \sim_\ZZ Q_2 \sim_\ZZ Q_3$, as desired.
\end{proof}

\subsection{\texorpdfstring{$a + b = 2$}{a + b = 2}}
In this case $(a, b) = (1, 1)$. Then the normalized rational 3-term linear relation is given by 
\begin{equation}
\label{eqn:half_half}
  \frac{1}{2}\uptheta_{Q_1} + \frac{1}{2}\uptheta_{Q_2} = \uptheta_{Q_3}.
\end{equation}
We apply \nameref{algo:extended_refinement_algorithm} with $\mathtt{STOP\_SET} \defeq \Delta$ and $\mathtt{MAX\_ITERATIONS} \defeq 13$. The algorithm terminated after iteration $i = 6$, there were no remaining pairs $(T, P)$ since all $(T, P) \in \mathcal{S}_6$ satisfied $T \subseteq \Delta$. Therefore $Q_1 \sim_\ZZ Q_2 \sim_\ZZ Q_3$. 

\medskip
The following table summarizes the number of pairs $(T, P) \in \mathcal{S}_i$. 

\begin{center}
\begin{tabular}{ |c|c|c|c|c|c|c|c|}
  \hline
   & \multicolumn{7}{|c|}{State after iteration $i$}   \\
 \hline
 $i$ & $0$ & $1$ & $2$ & $3$  & $4$ & $5$ & $6$\\ 
 \hline 
 $\#\mathcal{S}_i$ & $1$ & $3$ & $9$ & $29$ & $58$ & $30$  & $0$ \\
 \hline
 $\# \{(T, P) \in \mathcal{S}_i \mid T \neq \emptyset\}$ & $1$ & $3$ & $9$ & $16$ & $6$ & $0$ & $0$ \\
 \hline
\end{tabular} 
\end{center}

\subsection{\texorpdfstring{$a + b = 3$}{a + b = 3}} 
In this case $(a, b) = (1, 2)$ or $(2, 1)$, and by symmetry it suffices to consider only $(a, b) = (1, 2)$. Then the normalized rational 3-term linear relation is given by 
\begin{equation}
\label{eqn:one_third_two_third}
\frac{1}{3}\uptheta_{Q_1} + \frac{2}{3}\uptheta_{Q_2} = \uptheta_{Q_3}.  
\end{equation}
We apply \nameref{algo:extended_refinement_algorithm} with $\mathtt{STOP\_SET} \defeq \Delta$ and $\mathtt{MAX\_ITERATIONS} \defeq 13$. The algorithm terminated after $\mathtt{MAX\_ITERATIONS}$ iterations. The following table summarizes the number of pairs $(T, P) \in \mathcal{S}_i$ after iteration $i$.
\begin{center}
\begin{tabular}{ |c|c|c|c|c|c|c|c|c|c|c|c|c|c|c|}
  \hline
   & \multicolumn{14}{|c|}{State after iteration $i$}   \\
 \hline
 $i$ & $0$ & $1$ & $2$ & $3$ & $4$  & $5$ & $6$ & $7$ & $8$ & $9$ & $10$ & $11$ & $12$ & $13$ \\ 
 \hline 
 $\#\mathcal{S}_i$ & $1$ & $3$ & $11$ & $21$ & $13$ & $24$ & $16$ & $48$ & $33$ & $57$ & $27$ & $77$ & $42$ & $287$   \\
 \hline 
$\# \{(T, P) \in \mathcal{S}_i \mid T \neq \emptyset\}$ & $1$ & $3$ & $5$ & $2$  & $1$ & $1$ & $1$ & $1$ & $1$ & $1$ & $1$ & $1$ & $3$ & $3$ \\ 
 \hline
\end{tabular} 
\end{center}

\bigskip
\noindent Analyzing the unique pair $(T, P) \in \mathcal{T}_4$ with $T \neq \emptyset$, we see that $T = \RR_{\ge 0} \vec{v}$ where \[\vec{v} \defeq (\nicefrac{1}{4}, \nicefrac{1}{4}, \nicefrac{1}{4}) \times (1, 1, 1) \times (\nicefrac{1}{4}, \nicefrac{3}{4}, 0) \in \mathcal{V} \times \mathcal{V} \times \mathcal{V}.\] This gives us a potential candidate for a non-trivial solution up to scaling and equivalence, namely
\[
        Q_1 \defeq x^2 + xy + y^2, \qquad 
        Q_2 \defeq 4(x^2 + xy + y^2),  \qquad 
        Q_3 \defeq x^2 + 3y^2. 
\]
 We now prove that this candidate indeed satisfies \autoref{eqn:one_third_two_third}.

\begin{lemma}[\textbf{Non-Trivial $3$-term Linear Relation}]
\label{lemma:non_trivial_relation_theorem}
    Suppose $Q_1, Q_2, Q_3$ are positive-definite integer-valued binary quadratic forms such that 
\[
        Q_1 \sim_\ZZ c(x^2 + xy + y^2), \qquad 
        Q_2 \sim_\ZZ 4c(x^2 + xy + y^2), \quad \text{ and } \quad 
        Q_3 \sim_\ZZ c(x^2 + 3y^2) 
\]
for some positive integer $c \in \Z_{>0}$. Then we have the following relation among their theta series: 
        \[\frac{1}{3} \uptheta_{Q_1} + \frac{2}{3} \uptheta_{Q_2} = \uptheta_{Q_3}.\]
\end{lemma}

\begin{proof}
Let $(a, b) \in \Z^{\oplus 2}$, and let \[m \defeq a^2 + 3b^2.\]
It suffices to show that \[\frac{1}{3} r_{Q_1}(m) + \frac{2}{3} r_{Q_2}(m) = r_{Q_3}(m).\]

\smallskip
\textbf{Case 1.} If $a \equiv b \pmod 2$, then the map $(a, b) \mapsto \left(\frac{a - b}{2}, b \right)$ is a bijection from  $\rep_{Q_3}(m)$ to $\rep_{Q_2}(m)$, with the inverse map given by $(u, v) \mapsto (2u + v, v)$. Therefore $r_{Q_2}(m) = r_{Q_3}(m)$. Similarly, the map $(a, b) \mapsto (a-b, 2b)$ is a bijection from $\rep_{Q_3}(m)$ to $\rep_{Q_1}(m)$, with the inverse map given by $(u, v) \mapsto \left(u + \frac{v}{2}, \frac{v}{2}\right)$ (which is an integral vector since $u^2 + uv + v^2 \equiv 0 \pmod 2 \implies u, v \equiv 0 \pmod 2$). Therefore $r_{Q_1}(m) = r_{Q_2}(m) = r_{Q_3}(m)$, which suffices.

\smallskip
\textbf{Case 2.} If $a \not\equiv b \pmod 2$, then $r_{Q_2}(m) = 0$ since $Q_2(x, y) \in 2\ZZ$ when $x, y \in \ZZ$. To show that $r_{Q_1}(m) = 3r_{Q_3}(m)$, we consider the map $\varphi$ from $\rep_{Q_3}(m)$ to subsets of $\rep_{Q_1}(m)$ given by \[\varphi \: : \:(a, b) \mapsto \{(a-b, 2b), (-a+b, a+b), (2b, -a+b)\},\] 
which we show has the following two properties:
\begin{enumerate}
    \item[1)] $\varphi$ is injective. 
    \item[2)] For any $(u, v) \in \rep_{Q_1}(m)$ there exists some $(a, b) \in \rep_{Q_3(m)}$ so that $(u, v) \in \varphi((a, b))$.
\end{enumerate}
 To show part 1), it suffices to construct an inverse map $\psi$ so that $\psi(\varphi((a, b))) = (a, b)$. Suppose $\varphi((a, b)) = \{(u_1, v_1), (u_2, v_2), (u_3, v_3)\}$. Since $a \not\equiv b \pmod 2$, from the definition of $\varphi$ we see that exactly one of the elements $(u_i, v_i)$ where $1 \le i \le 3$ satisfies $v_i \equiv 0 \pmod 2$. Then we define the inverse map $\psi$ by 
 \[\psi \: : \: \{(u_1, v_1), (u_2, v_2), (u_3, v_3)\} \mapsto \left(u_i + \frac{v_i}{2}, \frac{v_i}{2}\right),\]
which we see is in $\rep_{Q_1}(m)$ and also satisfies $\psi(\varphi((a, b))) = (a, b)$.

\medskip
 Next we show part 2). Note that the set $\varphi((a, b))$ is invariant under the left-multiplication action $(u, v) \mapsto g \cdot (u, v)$ for all $g$ in the cyclic group \[G \defeq \left\{\begin{bmatrix}
    1 & 0 \\
    0 & 1 \\ 
\end{bmatrix}, \begin{bmatrix}
    -1 & 0 \\
    1 & 1 \\
\end{bmatrix}, \begin{bmatrix}
    0 & 1 \\ 
    -1 & 0 \\
\end{bmatrix}\right\}.\] 
Suppose  $(u, v) \in \rep_{Q_1}(m)$. Then define \[F \defeq G \cdot (u, v) = \{(u, v), (-u, u + v), (v, -u)\}.\] We now show that there exists some $(a, b) \in \rep_{Q_3}(m)$ so that $\varphi((a, b)) = F$. Consider the set \[H \defeq \begin{bmatrix}
    1 & \nicefrac{1}{2} \\
    0 & \nicefrac{1}{2} \\
\end{bmatrix} \cdot F = \left\{\left(u + \frac{v}{2}, \frac{v}{2} \right), \left(\frac{v - u}{2},  \frac{u + v}{2}\right), \left(v - \frac{u}{2}, -\frac{u}{2} \right)\right\}.\]
Each element $(x, y) \in H$  satisfies $x^2 + xy + y^2 = m$. However, since $m \equiv 1 \pmod 2$,  both $u$ and $v$ cannot be even.  So exactly one of the elements $(a, b) \in H$ satisfies $(a, b) \in \Z^{\oplus 2}$, giving $\varphi((a, b)) = F$. 

Together, parts 1) and 2) imply $r_{Q_1}(m) = 3 r_{Q_3}(m)$. Since $r_{Q_2}(m) = 0$, this proves the result.  
\end{proof}

\begin{remark}[\textbf{Alternative proof using Modular Forms}]
While the proof of \autoref{lemma:non_trivial_relation_theorem} is given in terms of an explicit construction 
on the representation vectors for the given quadratic forms, by \cite[Theorem 10.9, p175]{MR1474964} we could also verify this identity 
by a (less insightful) finite computation with the associated theta series $\uptheta_{Q_i}(z)$ as modular forms in the space $M_k(N, \chi) = M_1(12, \chi_{-3})$ 
where $\chi_{-3}: (\ZZ/3\ZZ)^\times \rightarrow \{\pm1\}$ is the non-trivial Dirichlet character.  Here identities between modular forms can be verified by checking that all Fourier coefficients agree up to the Sturm bound, which by \cite[Corollary 9.20, p174]{MR2289048} (applied for infinitely many prime ideals $\mathfrak{m} = p\ZZ$) requires us to verify the identity 
$$
\frac{1}{3} r_{Q_1}(m) + \frac{2}{3} r_{Q_2}(m) = r_{Q_3}(m)
$$
for m = 0, 1, and 2.  This holds, which proves the lemma.
\end{remark}

\subsection{\texorpdfstring{$a+b \ge 4$}{a+b >= 4}}
\label{sec:a_plus_b_ge_4}
In this case we show that there are no non-trivial solutions.  

\medskip
\begin{lemma}
For all $a+b \ge 4$, \nameref{algo:extended_refinement_algorithm} can be represented using the following refinement diagram: 

\begin{center}
\begin{equation}
\label{figure:ref_tree}
\begin{tikzcd}[sep=huge]
\emptyset & T_0 \arrow[l, "{(a+b, 0, a)}"'] \arrow[l, "\vdots"'  , bend left] \arrow[l, bend left=49] \arrow[r, "{(0, a+b, b)}"] \arrow[r, "\vdots", bend right] \arrow[r, bend right=49] \arrow[d, "{(1, 1, 1)}" description]          & \emptyset \\
\emptyset & T_1 \arrow[l, "{(a+b, 0, a)}"'] \arrow[l, "\vdots"' , bend left] \arrow[l, bend left=49] \arrow[r, "{(0, a+b, b)}"] \arrow[r, "\vdots", bend right] \arrow[r, bend right=49] \arrow[d, "{(1, 1, 1)}" description]          & \emptyset \\
\emptyset & T_2 \arrow[l, "{(a+b, 0, a)}"'] \arrow[l, "\vdots"', bend left] \arrow[l, bend left=49] \arrow[r, "{(0, a+b, b)}"] \arrow[r, "\vdots", bend right] \arrow[r, no head, bend right=49] \arrow[d, "{(1, 1, 1)}" description] & \emptyset \\
          & T_3 \subseteq \Delta,                                                                                                                                                                                                                                  &          
\end{tikzcd}
\end{equation}
\end{center}
where $T_0, T_1, T_2, T_3$ are non-empty polyhedral cones independent of $a$ and $b$, with $(T_i, P_i) \in \mathcal{S}_i$ for some covering parameter $P_i$, and also $T_3 \subseteq \Delta$. In this diagram, the refinement algorithm is represented as a rooted edge-labelled directed graph. The root node $T_0 = {\mathcal{V}} \times {\mathcal{V}} \times {\mathcal{V}}$ represents the initial state of the algorithm. For any node $T$,  its children are given by the refinements $T_{\cX, \cY, \cZ}$ across all possible $\cX, \cY, \cZ$ from \autoref{def:refinement_procedure}. The edge from $T$ to $T_{\cX, \cY, \cZ}$ is labelled $(|\cX|,|\cY|,|\cZ|) \in \cL_{a, b}$.
\end{lemma}

\begin{proof} We first prove the claim for $a+b = 4$, and then generalize it for all $a+b \ge 4$. 

\medskip
 \textbf{Case 1 ($a + b = 4$)}:  Since $\gcd(a, b) = 1$, we have either $(a, b) = (1, 3)$ or $(3, 1)$.  Now by applying \linebreak \nameref{algo:extended_refinement_algorithm}, we see that there exist $T_0, T_1, T_2, T_3$ (with $T_3 \subseteq \Delta$) so in either case the refinement algorithm is represented by the refinement diagram  (\hyperref[figure:ref_tree]{15}). (See \autoref{subsec:p1p2p3p4} for explicit polyhedral cone descriptions of $T_0, T_1, T_2, T_3$ and their respective covering parameters $P_0, P_1, P_2, P_3$). 

\medskip
 \textbf{Case 2 ($a + b > 4$)}: We use Case 1 to show that for all $a + b > 4$, the refinement algorithm is represented by the refinement diagram  (\hyperref[figure:ref_tree]{15}).  When an edge is labelled $(1, 1, 1)$, the refinement is independent of the choices of $a$ and $b$, and so the path $T_0 \xrightarrow[]{(1, 1, 1)} T_1 \xrightarrow[]{(1, 1, 1)} T_2 \xrightarrow[]{(1, 1, 1)} T_3$ is always in the directed graph. Therefore it suffices to show that any edge from $T_i \in \{T_0, T_1, T_2\}$ that is labelled $(a+b, 0, a)$ or $(0, a+b, b)$ is directed to an empty polyhedral cone.      

In \autoref{def:refinement_procedure}, if $\vec{n} = (a+b, 0, a)$ then $\mathcal{X} = \mathcal{X}(\vec{n}) \in \operatorname{MIN}_{a+b}\left(\mathbb{Z}_{\tilde{*}}^{\oplus 2} \setminus\left(X_1\union \ldots\union X_k\right)\right)$,  and if $\vec{n} = (0, a+b, b)$ then $\mathcal{Y} = \mathcal{Y}(\vec{n}) \in \operatorname{MIN}_{a+b}\left(\mathbb{Z}_{\tilde{*}}^{\oplus 2} \setminus\left(Y_1\union \ldots\union Y_k\right)\right)$.  We show  in these cases, that respectively either $\kset(\Z^{\oplus 2}_{\tilde{*}}; X_1, \ldots, X_k, \cX)$ or $\kset(\Z^{\oplus 2}_{\tilde{*}}; Y_1, \ldots, Y_k, \cY)$ is the zero-cone (i.e. the polyhedral cone $\{(0, 0, 0)\}$). Since the binary quadratic form given by $Q = 0$ is not in $\mathcal{V}$, the refinement of $(T_i, P_i)$ will be empty. Below, we execute the refinement for $(T_i, P_i)$ with $T_i \in \{T_0, T_1, T_2\}$.   

\medskip
 \textbf{Case 2a (Refining $(T_0, P_0)$)}. $T_0$ has covering parameter $P_0 = ((), (), ())$, where $()$ denotes an  empty sequence. Since $a + b > 4$, for any $X_1 \in \minimum_{a+b}(\Z^{\oplus 2}_{\tilde{*}})$, there exists $X'_1 \in \minimum_{4}(\Z^{\oplus 2}_{\tilde{*}})$ so that $X'_1 \subseteq X_1$. Therefore $\kset(\Z^{\oplus 2}_{\tilde{*}}; X_1) \subseteq \kset(\Z^{\oplus 2}_{\tilde{*}}; X'_1)$. We will show that $\kset(\Z^{\oplus 2}_{\tilde{*}}; X'_1)$ is the zero-cone, implying $\kset(\Z^{\oplus 2}_{\tilde{*}}, X_1)$ is also the zero-cone.    

Since $\minimum_4(\Z_{\tilde{*}}^{\oplus 2}) = \{\{(1, 0), (0, 1), (-1, 1), (1, 1)\}\}$, we can use our $\kset$-set algorithm described in \autoref{lemma:Kset} to obtain 
\[\kset(\Z^{\oplus 2}_{\tilde{*}}; \{(1, 0), (0, 1), (-1, 1), (1, 1)\}) = \text{{Polyhedral Cone }} \{(0, 0, 0)\},\]
as required.  
After refining $(T_0, P_0)$, our refinement diagram (\hyperref[figure:ref_tree]{15}) now looks like this: 
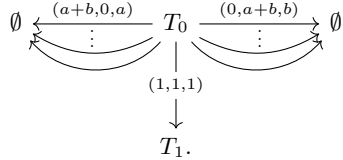
\begin{center}
\begin{tikzcd}[sep=huge]
\emptyset & T_0 \arrow[l, "{(a+b, 0, a)}"'] \arrow[l, "\vdots"', bend left] \arrow[l, bend left=49] \arrow[r, "{(0, a+b, b)}"] \arrow[r, "\vdots", bend right] \arrow[r, bend right=49] \arrow[d, "{(1, 1, 1)}" description]          & \emptyset \\
& T_1.   
\end{tikzcd}
\end{center}

\textbf{Case 2b (Refining $(T_1, P_1)$)}. $T_1$ has covering parameter $P_1 = ((\{((1, 0))\}), (\{((1, 0))\}), (\{((1, 0))\}))$. For any $X_2 \in \minimum_{a+b}(\Z^{\oplus 2}_{\tilde{*}}\setminus X_1)$ where $X_1 \defeq  \{(1, 0)\}$, there exists some $X'_2 \in \minimum_{4}(\Z^{\oplus 2}_{\tilde{*}}\setminus X_1)$ so that $X'_2 \subseteq X_2$. As in Case 2a, we show that  $\kset(\Z^{\oplus 2}_{\tilde{*}}; X_1, X'_2)$ is the zero-cone. Since $\minimum_4(\Z^{\oplus 2}_{\tilde{*}} \setminus \{(1, 0)\}) = \{\{(0, 1), (-1, 1), (1, 1), (-2, 1)\}\}$, we again use our $K$-set algorithm in \autoref{lemma:Kset} to obtain 
\[\kset(\Z^{\oplus 2}_{\tilde{*}}; \{(1, 0)\}, \{(0, 1), (-1, 1), (1, 1), (-2, 1)\}) = \text{{Polyhedral Cone }} \{(0, 0, 0)\},\]   
as required. After refining $(T_1, P_1)$, our refinement diagram (\hyperref[figure:ref_tree]{15})  now looks like this:
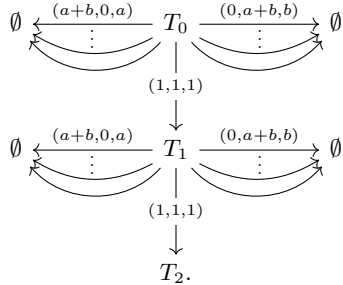
\begin{center}
\begin{tikzcd}[sep=huge]
\emptyset & T_0 \arrow[l, "{(a+b, 0, a)}"'] \arrow[l, "\vdots"', bend left] \arrow[l, bend left=49] \arrow[r, "{(0, a+b, b)}"] \arrow[r, "\vdots", bend right] \arrow[r, bend right=49] \arrow[d, "{(1, 1, 1)}" description]          & \emptyset \\
\emptyset & T_1 \arrow[l, "{(a+b, 0, a)}"'] \arrow[l, "\vdots"', bend left] \arrow[l, bend left=49] \arrow[r, "{(0, a+b, b)}"] \arrow[r, "\vdots", bend right] \arrow[r, bend right=49] \arrow[d, "{(1, 1, 1)}" description]          & \emptyset \\
& T_2. &      
\end{tikzcd}
\end{center} 
 
 \textbf{Case 2c (Refining $(T_2, P_2)$)}. $T_2$ has covering parameter \[P_2 = ((\{(1, 0)\}, \{(0, 1)\}), (\{(1, 0)\}, \{(0, 1)\}), (\{(1, 0)\}, \{(0, 1)\})).\] For any $X_3 \in \minimum_{a+b}(\Z^{\oplus 2}_{\tilde{*}}\setminus (X_1 \union X_2))$ where $X_1 \defeq  \{(1, 0)\}$ and $X_2 \defeq  \{(0, 1)\}$, there exists some $X'
_3 \in \minimum_{4}(\Z^{\oplus 2}_{\tilde{*}}\setminus (X_1 \union X_2))$ so that $X'_3 \subseteq X_3$. As in the previous two cases, we show that $\kset(\Z^{\oplus 2}_{\tilde{*}}; X_1, X_2, X'_3)$ is the zero-cone. Since \[\minimum_4(\Z^{\oplus 2}_{\tilde{*}} \setminus \{(1, 0), (0, 1)\}) = \{\{(-1, 1), (1, 1), (-2, 1), (2, 1)\}, \{(-1, 1), (1, 1), (-2, 1), (-1, 2)\} \},\] there are two choices of $X'_3$. If $X'_3 = \{(-1, 1), (1, 1), (-2, 1), (2, 1)\}$, then 
\[\kset(\Z^{\oplus 2}_{\tilde{*}}; \{(1, 0)\}, \{(0, 1)\}, \{(-1, 1), (1, 1), (-2, 1), (2, 1)\}) = \text{{Polyhedral Cone  }} \{(0, 0, 0)\},\]
and if $X'_3 = \{(-1, 1), (1, 1), (-2, 1), (-1, 2)\}$, then 
\[\kset(\Z^{\oplus 2}_{\tilde{*}}; \{(1, 0)\}, \{(0, 1)\}, \{(-1, 1), (1, 1), (-2, 1), (-1, 2)\}) = \text{{Polyhedral Cone }} \{(0, 0, 0)\},\]
as required.  After refining $(T_2, P_2)$, the refinement diagram (\hyperref[figure:ref_tree]{15})  is complete: 

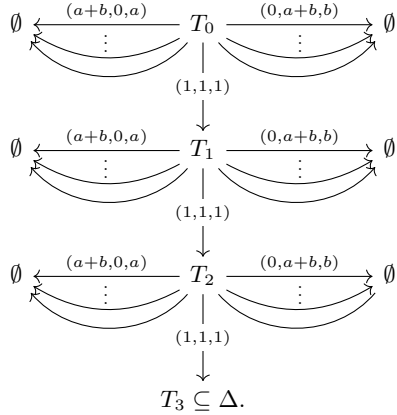
\begin{center}
\begin{tikzcd}[sep=huge]
\emptyset & T_0 \arrow[l, "{(a+b, 0, a)}"'] \arrow[l, "\vdots"'  , bend left] \arrow[l, bend left=49] \arrow[r, "{(0, a+b, b)}"] \arrow[r, "\vdots", bend right] \arrow[r, bend right=49] \arrow[d, "{(1, 1, 1)}" description]          & \emptyset \\
\emptyset & T_1 \arrow[l, "{(a+b, 0, a)}"'] \arrow[l, "\vdots"' , bend left] \arrow[l, bend left=49] \arrow[r, "{(0, a+b, b)}"] \arrow[r, "\vdots", bend right] \arrow[r, bend right=49] \arrow[d, "{(1, 1, 1)}" description]          & \emptyset \\
\emptyset & T_2 \arrow[l, "{(a+b, 0, a)}"'] \arrow[l, "\vdots"', bend left] \arrow[l, bend left=49] \arrow[r, "{(0, a+b, b)}"] \arrow[r, "\vdots", bend right] \arrow[r, no head, bend right=49] \arrow[d, "{(1, 1, 1)}" description] & \emptyset \\
          & T_3 \subseteq \Delta.                                                                                                                                                                                                                                  &          
\end{tikzcd}
\end{center}
\end{proof}

\section{Acknowledgements}

We would like to thank Sabrina Reguyal for translating Schiemann's original paper from German. This was invaluable for understanding Schiemann's complete perspective on these computations.

\section{Index of Terminology and Notation}

\begin{center}
\begin{tabular}{|c|c|c|c|}
\hline
\textbf{Terminology} & \textbf{Notation} &  \vphantom{$\dfrac{1}{2}$} \textbf{Defined in} \\
\hline \hline
Admissibility of Covering Parameter & --- & \vphantom{$\dfrac{1}{2}$} \autoref{def:covering_admissibility} \\
\hline
Covering & --- & \vphantom{$\dfrac{1}{2}$}\autoref{def:covering}\\
\hline
Covering Parameter & $((X_i)_{i=1}^{k}, (Y_i)_{i=1}^{k}, (Z_i)_{i=1}^{k})$ &\vphantom{$\dfrac{1}{2}$} \autoref{def:covering_parameter} \\
\hline 
Covering Parameter Auxiliary Sets & $\mathcal{K}_{\cX, \cY, \cZ}, \mathcal{Q}_{\cX, \cY, \cZ} $ & \vphantom{$\dfrac{1}{2}$}\autoref{def:auxiliary_sets} \\ 
\hline 
Diagonal & $\Delta$ & \vphantom{$\dfrac{1}{2}$}\autoref{def:solution_set} \\ 
\hline 
Edges of $\mathcal{V}$ & $\cE_1, \cE_2, \cE_3$ &\vphantom{$\dfrac{1}{2}$} \autoref{rem:edges_of_V} \\
\hline 
Extended Refinement Algorithm & --- &\vphantom{$\dfrac{1}{2}$} \hyperref[algo:extended_refinement_algorithm]{Algorithm 1} \\
\hline 
Extended Refinement Algorithm Sets & $\mathcal{S}_i$ & \vphantom{$\dfrac{1}{2}$} \hyperref[algo:extended_refinement_algorithm]{Algorithm 1} \\
\hline
$\text{GL}_2(\ZZ)$-Reduced Positive-Definite Binary Quadratic Forms & $\mathcal{V}$ & \vphantom{$\dfrac{1}{2}$} \autoref{def:reduced_forms}\\
\hline 
Integral Form of Normalized Rational $3$-Term Linear Relation & --- & \vphantom{$\dfrac{1}{2}$} \autoref{remark:integral_form} \\
\hline
$\kset$-set & $\kset(X; X_1, \ldots, X_k)$ & \vphantom{$\dfrac{1}{2}$} \autoref{def:kset_set} \\
\hline
Linset & $\cL_{a, b}$ & \vphantom{$\dfrac{1}{2}$} \autoref{def:linset} \\
\hline 
Max Iterations & $\mathtt{MAX\_ITERATIONS}$ & \vphantom{$\dfrac{1}{2}$} \hyperref[algo:extended_refinement_algorithm]{Algorithm 1} \\
\hline
Minimal Subset & $\minimum(X)$ & \vphantom{$\dfrac{1}{2}$} \autoref{def:minimum} \\
\hline 
$n^{\text{th}}$-order Minimal Subset & $\minimum_n(X)$ & \vphantom{$\dfrac{1}{2}$} \autoref{def:nth_order_minimal_subset} \\
\hline 
Normalized Rational $3$-Term Linear Relation & --- & \vphantom{$\dfrac{1}{2}$} \autoref{def:normalized_rational} \\
\hline
Ordering Relation(s)  & $\preceq$, $\succeq$ & \vphantom{$\dfrac{1}{2}$} \autoref{def:pre_order} \\
\hline
Polyhedral Cone & $\cP(A, B)$ & \vphantom{$\dfrac{1}{2}$} \autoref{def:polycone}\\
\hline
Refinement of a Pair & $\mathcal{M}_{a, b}$ & \vphantom{$\dfrac{1}{2}$} \autoref{def:refinement_procedure} \\
\hline 
Refinement of Set of Pairs & $\mathtt{refine}_{a, b}$ & \vphantom{$\dfrac{1}{2}$} \autoref{def:refinement_procedure}\\ 
\hline 
Refinement Process & --- & \vphantom{$\dfrac{1}{2}$} \autoref{def:refinement_procedure} \\
\hline
Refinement Sequence & $(\mathcal{T}_i)_{i=0}^{\infty}$ & \vphantom{$\dfrac{1}{2}$} \autoref{def:refinement_sequence} \\
\hline
Representation Number & $r_Q(m)$ &  \vphantom{$\dfrac{1}{2}$}\autoref{sec:introduction}\\
\hline
Representation Set & $\rep_Q(m)$ &  \vphantom{$\dfrac{1}{2}$}\autoref{sec:introduction}\\
\hline
Solution Set of \autoref{eqn:rational_lin} & $\mathcal{D}_{a, b}$ & \vphantom{$\dfrac{1}{2}$} \autoref{def:solution_set} \\
\hline 
Stop Set & $\mathtt{STOP\_SET}$ &  \vphantom{$\dfrac{1}{2}$}\hyperref[algo:extended_refinement_algorithm]{Algorithm 1} \\
\hline
Strongly Primitive Representation Number & \vphantom{$\dfrac{1}{2}$} $r^{\tilde{*}}_Q(m)$ & \autoref{def:strongly_primitive_objects}\\
\hline
Strongly Primitive Representation Set & \vphantom{$\dfrac{1}{2}$} $\rep^{\tilde{*}}_Q(m)$ & \autoref{def:strongly_primitive_objects}\\
\hline
Strongly Primitive Theta Series & \vphantom{$\dfrac{1}{2}$} $\uptheta_Q^{\tilde{*}}(z)$ & \autoref{def:strongly_primitive_objects}\\
\hline
Strongly Primitive Vectors & $\Z^{\oplus 2}_{\tilde{*}}$ & \vphantom{$\dfrac{1}{2}$} \autoref{def:strongly_primitive_objects}\\
\hline 
Successive Minima Sequence & --- & \vphantom{$\dfrac{1}{2}$} \autoref{def:succ_min_set}\\
\hline
Theta Series & $\uptheta_Q(z)$ &  \vphantom{$\dfrac{1}{2}$} \autoref{sec:introduction}\\
\hline
$\ZZ$-equivalence & $\sim_\ZZ$ & \vphantom{$\dfrac{1}{2}$} \autoref{sec:introduction} \\
\hline
\end{tabular}
\end{center}

\newpage
\bibliographystyle{plain} 
\bibliography{references}

\begin{thebibliography}{1}

\bibitem{buell1989binary}
Duncan~A Buell.
\newblock {\em Binary quadratic forms: classical theory and modern
  computations}.
\newblock Springer Science \& Business Media, 1989.

\bibitem{MR1474964}
Henryk Iwaniec.
\newblock {\em Topics in classical automorphic forms}, volume~17 of {\em
  Graduate Studies in Mathematics}.
\newblock American Mathematical Society, Providence, RI, 1997.

\bibitem{centrallyAntisymmetricPolyCone}
{Jonathan Hanke, Rahul Saha}.
\newblock {\em SageMath Package for Centrally Antisymmetric Polyhedral Cone}.
\newblock Unpublished.

\bibitem{MR4520776}
Erik Nilsson, Julie Rowlett, and Felix Rydell.
\newblock The isospectral problem for flat tori from three perspectives.
\newblock {\em Bull. Amer. Math. Soc. (N.S.)}, 60(1):39--83, 2023.

\bibitem{rydell2020three}
Felix Rydell.
\newblock Three perspectives of schiemann’s theorem.
\newblock {\em Master's Thesis, Department of Mathematical Sciences, Chalmers
  University of Technology, Gothenburg's University}.

\bibitem{schiemann1997ternary}
Alexander Schiemann.
\newblock Ternary positive definite quadratic forms are determined by their
  theta series.
\newblock {\em Mathematische Annalen}, 308(3):507--517, 1997.

\bibitem{MR2289048}
William Stein.
\newblock {\em Modular forms, a computational approach}, volume~79 of {\em
  Graduate Studies in Mathematics}.
\newblock American Mathematical Society, Providence, RI, 2007.
\newblock With an appendix by Paul E. Gunnells.

\bibitem{sagemath}
{The Sage Developers}.
\newblock {\em {S}ageMath, the {S}age {M}athematics {S}oftware {S}ystem
  ({V}ersion 9.5)}, 2022.
\newblock \href{https://www.sagemath.org}{\tt https://www.sagemath.org}.

\end{thebibliography}
\medskip
\textsc{Rahul Saha, New York, NY, USA}\\ 
\textit{E-mail address}: \texttt{rah4927@gmail.com} \\ 
\textit{URL}: \href{rahulsaha.net}{\texttt{http://www.rahulsaha.net}} \\ 

\noindent \textsc{Jonathan Hanke, Princeton, NJ 08542, USA} \\ 
\textit{E-mail address}: \texttt{jonhanke@gmail.com} \\
\textit{URL}: \href{jonhanke.com}{\texttt{http://www.jonhanke.com}} \\ 

\pagebreak 
    
\section{Appendix} 
\label{sec:appendix} 

\subsection{Example \texorpdfstring{$\minimum$}{MIN} Computation}

\noindent {\large \textbf{Example 1.}} We will compute $\minimum(\Z^{\oplus 2}_{\tilde{*}})$ using \autoref{lemma:min_properties}. First of all, note that for $a = 1$, $(a, 1) \in \Z^{\oplus 2}_{\tilde{*}}$. Then, we define $W_{0, 1} = \{x \in \Z^{\oplus 2}_{\tilde{*}} : x \not\succeq (1, 1)\} \union \{(a, 1)\}$. By the lemma, $W_{0, 1} \subseteq \{\vx \in \Z^{\oplus 2}_{\tilde{*}} : ||x||_{\infty} \le \sqrt{6}\}$ and the latter is a finite set. Then, iterating gives us $W_{0, 1} = \{(-1, 1), (0, 1), (1, 0), (1, 1)\}$ We know from the lemma that $\minimum(\Z^{\oplus 2}_{\tilde{*}}) = \minimum(\Z^{\oplus 2}_{\tilde{*}} \intersect W_{0, a}) = \minimum(\{(-1, 1), (0, 1), (1, 0), (1, 1)\})$. 

This is much better because now we have to compute the $\minimum$ of a finite set. For each element in this set, we can check using \autoref{lemma:preceq_algo} if any other element of the set precedes it. If not, then it is in the $\minimum$. Doing this gets us $\minimum(\Z^{\oplus 2}_{\tilde{*}}) = \{(1, 0)\}$.

\medskip
\noindent {\large \textbf{Example 2.}} We will compute $\minimum_n(\Z^{\oplus 2}_{\tilde{*}})$ for $n = 6$. We will do this inductively. We know from the previous example that \[\minimum_1(\Z^{\oplus 2}_{\tilde{*}}) = \{\{(1, 0)\}\}\] Now, we can use \autoref{lemma:min_properties} to obtain \[\minimum(\Z^{\oplus 2}_{\tilde{*}} \setminus \{(1, 0)\}) = \{(0, 1)\} \] 
and thus \[\minimum_2(\Z^{\oplus 2}_{\tilde{*}}) = \{\{(1, 0), (0, 1)\}\}\]
Similarly, since \[\minimum(\Z^{\oplus 2}_{\tilde{*}} \setminus \{(1, 0), (0, 1)\}) = \{(-1, 1)\}\] we get \[\minimum_3(\Z^{\oplus 2}_{\tilde{*}}) = \{\{(1, 0), (0, 1), (-1, 1)\}\}\] This continues for $n = 4$ and $n = 5$,  \[\minimum_4(\Z^{\oplus 2}_{\tilde{*}}) = \{\{(1, 0), (0, 1), (-1, 1), (1, 1)\}\}\] \[\minimum_5(\Z^{\oplus 2}_{\tilde{*}}) = \{\{(1, 0), (0, 1), (-1, 1), (1, 1), (-2, 1)\}\}.\] The next step however is a little different. 
This time, we see that there are two elements in the minima. \[\minimum(\Z^{\oplus 2}_{\tilde{*}} \setminus \{(1, 0), (0, 1), (-1, 1), (1, 1), (-2, 1)\}) = \{(2, 1), (-1, 2)\} \]  This gives us two sets in $\minimum_6$, \[\minimum_6(\Z^{\oplus 2}_{\tilde{*}}) = \{ \{(1, 0), (0, 1), (-1, 1), (1, 1), (-2, 1), (2, 1)\},  \{(1, 0), (0, 1), (-1, 1), (1, 1), (-2, 1), (-1, 2)\}\}.\]

\subsection{Example \texorpdfstring{$K$}{K}-set Computation}
\label{subsec:example_kset}

Below, we will adopt the convention to represent sets $A$ and $B$ while defining a polyhedral cone $\cP(A, B)$ as matrices. As an example, the set $\{(1, 2, 1), (-1, 1, 0)\}$ becomes $\begin{bmatrix}
    1 & 2 & 1 \\ 
    -1 & 1 & 0 \\ 
\end{bmatrix}$. Note that the ordering of the rows is irrelevant.

\medskip

\noindent {\large \textbf{Example.}} We will compute $\kset(\Z^{\oplus 2}_{\tilde{*}}; \{(1, 0), (0, 1)\}, \{\}, \{(-1, 1)\})$. We will use \autoref{lemma:Kset}. Suppose a quadratic form $Q \defeq (Q_{11}, Q_{22}, Q_{12})$ is in $K(\Z^{\oplus 2}_{\tilde{*}}; \{(1, 0), (0, 1)\}, \{\}, \{(-1, 1)\})$. Then, for our lemma to work, we will start by ignoring the empty sets, and picking a representative element from each set. Suppose we choose $(1, 0)$ and $(-1, 1)$. Then, we will compute $\kset(\Z^{\oplus 2}_{\tilde{*}}; (1, 0), (0, 1) (-1, 1))$ using \autoref{lemma:Kset_schiemann}. Since $\minimum(\Z^{\oplus 2}_{\tilde{*}} \setminus \{(1,0), (0,1), (-1, 1)\}) = \{(1, 1)\}$, we need \[Q(1, 0) \le Q(0, 1) \le Q(-1, 1) \le Q(1, 1)\]
\[\kset(\Z^{\oplus 2}_{\tilde{*}}; (1, 0), (0, 1) (-1, 1)) = \cP\left(\left[\begin{array}{rrr}
-1 & 1 & 0 \\
1 & 0 & -1 \\
0 & 0 & 2
\end{array}\right], [\quad] \right)\]
The second part is to compute the equalities, in our case, we only care about the equality constraint in $\{(1, 0), (0, 1)\}$, since all the other sets have fewer than two elements and so the equality conditions hold vacuously. We want 
\[Q(1, 0) = Q(0, 1)\]
Adding this condition to  $\kset(\Z^{\oplus 2}_{\tilde{*}}; (1, 0), (0, 1) (-1, 1))$ gives us 
\[ \kset(\Z^{\oplus 2}_{\tilde{*}}; \{(1, 0), (0, 1)\}, \{\}, \{(-1, 1)\}) = \cP\left(\left[\begin{array}{rrr}
-1 & 1 & 0 \\
1 & 0 & -1 \\
0 & 0 & 2 \\
1 & -1 & 0 \\
-1 & 1 & 0
\end{array}\right], [\quad]\right).\]

\subsection{Example Refinement Computation}
\label{sec:example_refinement_computation}
Suppose $(a, b) = (1, 2)$. Then $\cL_{a, b} = \cL_{1, 2} =  \{(1, 1, 1), (3, 0, 1), (0, 3, 2)\}$. Below, we will show two examples of refinement. Below, for each polyhedral cone, we will only provide the face description using $A$ and $B$. One can obtain the edge description by following any edge computing algorithm. 

\medskip
\noindent {\large \textbf{Example 1.}} Suppose we start with a polyhedral cone $P$ with the following description,  
\[P \left(\begin{bmatrix}
-1 & 1 & 0 & 0 & 0 & 0 & 0 & 0 & 0 \\
1 & 0 & -1 & 0 & 0 & 0 & 0 & 0 & 0 \\
0 & 0 & 1 & 0 & 0 & 0 & 0 & 0 & 0 \\
0 & 0 & 0 & 0 & 0 & 0 & -1 & 1 & 0 \\
0 & 0 & 0 & 0 & 0 & 0 & 1 & 0 & -1 \\
0 & 0 & 0 & 0 & 0 & 0 & 0 & 0 & 1 \\
0 & 0 & 0 & -1 & 1 & 0 & 0 & 0 & 0 \\
0 & 0 & 0 & 1 & 0 & -1 & 0 & 0 & 0 \\
0 & 0 & 0 & 0 & 0 & 1 & 0 & 0 & 0 \\
1 & 0 & 0 & 0 & 0 & 0 & 0 & 0 & 0 \\
0 & 0 & 0 & 1 & 0 & 0 & 0 & 0 & 0 \\
0 & 0 & 0 & 0 & 0 & 0 & 1 & 0 & 0
\end{bmatrix}, \quad  \begin{bmatrix}
1 & 0 & 0 & 0 & 0 & 0 & 0 & 0 & 0 \\
0 & 0 & 0 & 1 & 0 & 0 & 0 & 0 & 0 \\
0 & 0 & 0 & 0 & 0 & 0 & 1 & 0 & 0
\end{bmatrix} \right)\]
with covering parameter $((), (), ())$. Note this implies that we have not refined this polyhedral cone yet. Indeed, the above polyhedral cone corresponds to $\overline{V} \times \overline{V} \times \overline{V}$. Now, we will refine the polyhedral cone for each vector $\vec{n}$ in $\cL_{1, 2}$. 

\noindent \underline{$\vec{n} = (1, 1, 1)$}. Then, we have 
\begin{equation*}
    \begin{split}
    \minimum_1(\Z^{\oplus 2}_{\tilde{*}}) = \{\{(1, 0)\}\}\\ 
    \minimum_1(\Z^{\oplus 2}_{\tilde{*}}) = \{\{(1, 0)\}\}\\ 
    \minimum_1(\Z^{\oplus 2}_{\tilde{*}}) = \{\{(1, 0)\}\}\\
    \end{split}
\end{equation*}
Now, we will iterate over each triplet $(\mathcal{X}, \mathcal{Y}, \mathcal{Z}) \in \{\{(1, 0)\}\} \times \{\{(1, 0)\}\} \times \{\{(1, 0)\}\}$. The only such triplet is $(\mathcal{X}, \mathcal{Y}, \mathcal{Z}) = (\{(1, 0)\}, \{(1, 0)\}, \{(1, 0)\})$. Then, 
\[\cK_{\cX,\cY,\cZ} = \cP\left(\begin{bmatrix}
    -1 & 1 & 0 & 0 & 0 & 0 & 0 & 0 & 0\\
    1 & -1 & 0 & 0 & 0 & 0 & 0 & 0 & 0 \\ 
     0 & 0 & 0 & -1 & 1 & 0 & 0 & 0 & 0\\
     0 & 0 & 0 & 1 & -1 & 0 & 0 & 0 & 0\\
    -1 & 1 & 0 & 0 & 0 & 0 & 0 & 0 & 0\\
    1 & -1 & 0 & 0 & 0 & 0 & 0 & 0 & 0\\
\end{bmatrix}, [\quad]\right),\]
\[\cQ_{\cX,\cY,\cZ} = \cP\left(\begin{bmatrix}
    1 & 0 & 0 & -1 & 0 & 0 & 0 & 0 & 0 \\
    -1 & 0 & 0 & 1 & 0 & 0 & 0 & 0 & 0 \\ 
    1 & 0 & 0 & 0 & 0 & 0 & -1 & 0 & 0 \\
    -1 & 0 & 0 & 0 & 0 & 0 & 1 & 0 & 0 \\ 
\end{bmatrix}, [\quad]\right).\]
Here, we obtained $\cK_{\cX,\cY,\cZ}$ by following the same steps as in \autoref{subsec:example_kset}. The $\cQ_{\cX,\cY,\cZ}$ was obtained directly from the definition, since we want 
\[Q_1(\{(1, 0)\}) = Q_2(\{(1, 0)\}) = Q_3(\{(1, 0)\})\]
which gives us two independent linear equalities. 

Finally, we get the refinement $P'$ as follows, \[P' = P \intersect \cK_{\cX,\cY,\cZ} \intersect \cQ_{\cX,\cY,\cZ}\] which obtains 
\[P' = \cP\left(\left[\begin{array}{rrrrrrrrr}
-1 & 1 & 0 & 0 & 0 & 0 & 0 & 0 & 0 \\
1 & 0 & -1 & 0 & 0 & 0 & 0 & 0 & 0 \\
0 & 0 & 1 & 0 & 0 & 0 & 0 & 0 & 0 \\
0 & 0 & 0 & 0 & 0 & 0 & -1 & 1 & 0 \\
0 & 0 & 0 & 0 & 0 & 0 & 1 & 0 & -1 \\
0 & 0 & 0 & 0 & 0 & 0 & 0 & 0 & 1 \\
0 & 0 & 0 & -1 & 1 & 0 & 0 & 0 & 0 \\
0 & 0 & 0 & 1 & 0 & -1 & 0 & 0 & 0 \\
0 & 0 & 0 & 0 & 0 & 1 & 0 & 0 & 0 \\
1 & 0 & 0 & 0 & 0 & 0 & -1 & 0 & 0 \\
-1 & 0 & 0 & 0 & 0 & 0 & 1 & 0 & 0 \\
0 & 0 & 0 & 1 & 0 & 0 & -1 & 0 & 0 \\
0 & 0 & 0 & -1 & 0 & 0 & 1 & 0 & 0 \\
1 & 0 & 0 & 0 & 0 & 0 & 0 & 0 & 0 \\
0 & 0 & 0 & 1 & 0 & 0 & 0 & 0 & 0 \\
0 & 0 & 0 & 0 & 0 & 0 & 1 & 0 & 0
\end{array}\right], \left[\begin{array}{rrrrrrrrr}
1 & 0 & 0 & 0 & 0 & 0 & 0 & 0 & 0 \\
0 & 0 & 0 & 1 & 0 & 0 & 0 & 0 & 0 \\
0 & 0 & 0 & 0 & 0 & 0 & 1 & 0 & 0
\end{array}\right]\right).\] 

\noindent \underline{$\vec{n} = (3, 0, 1)$}. Then, we have 
\begin{equation*}
    \begin{split}
    \minimum_3(\Z^{\oplus 2}_{\tilde{*}}) &= \{\{(1, 0), (0, 1), (-1, 1)\}\}\\ 
    \minimum_0(\Z^{\oplus 2}_{\tilde{*}}) &= \{\{\}\}\\ 
    \minimum_1(\Z^{\oplus 2}_{\tilde{*}}) &= \{\{(1, 0)\}\}\\
    \end{split}
\end{equation*}
Now, we will iterate over each triplet  $(\mathcal{X}, \mathcal{Y}, \mathcal{Z}) \in \{\{(1, 0), (0, 1), (-1, 1)\}\} \times \{\{\}\} \times \{\{(1, 0)\}\}$. The only such triplet is  $(\mathcal{X}, \mathcal{Y}, \mathcal{Z}) = (\{(1, 0), (0, 1), (-1, 1)\}, \{\}, \{(1, 0)\})$. Then, 
\[\cK_{\cX,\cY,\cZ} = \cP \left(\begin{bmatrix}
-1 & 1 & 0 & 0 & 0 & 0 & 0 & 0 & 0 \\
1 & 0 & -1 & 0 & 0 & 0 & 0 & 0 & 0 \\
0 & 0 & 2 & 0 & 0 & 0 & 0 & 0 & 0 \\
1 & -1 & 0 & 0 & 0 & 0 & 0 & 0 & 0 \\
-1 & 1 & 0 & 0 & 0 & 0 & 0 & 0 & 0 \\
0 & -1 & 1 & 0 & 0 & 0 & 0 & 0 & 0 \\
0 & 1 & -1 & 0 & 0 & 0 & 0 & 0 & 0 \\
0 & 0 & 0 & 0 & 0 & 0 & -1 & 1 & 0 \\
\end{bmatrix}, [\quad] \right),\]
\[\cQ_{\cX,\cY,\cZ} = \cP\left(\begin{bmatrix}
1 & 0 & 0 & 0 & 0 & 0 & -1 & 0 & 0 \\
-1 & 0 & 0 & 0 & 0 & 0 & 1 & 0 & 0 \\
0 & 1 & 0 & 0 & 0 & 0 & -1 & 0 & 0 \\
0 & -1 & 0 & 0 & 0 & 0 & 1 & 0 & 0 \\
1 & 1 & -1 & 0 & 0 & 0 & -1 & 0 & 0 \\
-1 & -1 & 1 & 0 & 0 & 0 & 1 & 0 & 0
\end{bmatrix} , [\quad]\right).\]
Here, we can obtain $\cQ_{\cX,\cY,\cZ}$ from the definition by noting that \[Q(\{(1, 0), (0, 1), (-1, 1)\}) = Q(\{(1, 0)\})\] which gives us $3$ independent equalities. However, there is an optimization that we can make here. Note that since the $K$-set already enforces equalities on each of $\{(1, 0), (0, 1), (-1, )\}, \{(1, 0)\}, \{\}$, we can in fact get away with just one equality given by $Q(1, 0) = Q(-1, 1)$. Thus, we can instead used 
\[\cP\left(\begin{bmatrix}
1 & 0 & 0 & 0 & 0 & 0 & -1 & 0 & 0 \\
-1 & 0 & 0 & 0 & 0 & 0 & 1 & 0 & 0 \\
\end{bmatrix} , [\quad]\right)\] as $\cQ_{\cX,\cY,\cZ}$. Indeed, for our implementation, that is what we do. 

Finally, we get the refinement $P'$ as follows, \[P' = P \intersect \cK_{\cX,\cY,\cZ} \intersect \cQ_{\cX,\cY,\cZ},\] which obtains \[P' = \cP\left(\left[\begin{array}{rrrrrrrrr}
-1 & 1 & 0 & 0 & 0 & 0 & 0 & 0 & 0 \\
1 & 0 & -1 & 0 & 0 & 0 & 0 & 0 & 0 \\
0 & 0 & 0 & 0 & 0 & 0 & -1 & 1 & 0 \\
0 & 0 & 0 & 0 & 0 & 0 & 1 & 0 & -1 \\
0 & 0 & 0 & 0 & 0 & 0 & 0 & 0 & 1 \\
0 & 0 & 0 & -1 & 1 & 0 & 0 & 0 & 0 \\
0 & 0 & 0 & 1 & 0 & -1 & 0 & 0 & 0 \\
0 & 0 & 0 & 0 & 0 & 1 & 0 & 0 & 0 \\
-1 & 0 & 1 & 0 & 0 & 0 & 0 & 0 & 0 \\
1 & 0 & 0 & 0 & 0 & 0 & -1 & 0 & 0 \\
-1 & 0 & 0 & 0 & 0 & 0 & 1 & 0 & 0 \\
-1 & -1 & 1 & 0 & 0 & 0 & 1 & 0 & 0 \\
1 & 0 & 0 & 0 & 0 & 0 & 0 & 0 & 0 \\
0 & 0 & 0 & 1 & 0 & 0 & 0 & 0 & 0 \\
0 & 0 & 0 & 0 & 0 & 0 & 1 & 0 & 0
\end{array}\right], \left[\begin{array}{rrrrrrrrr}
1 & 0 & 0 & 0 & 0 & 0 & 0 & 0 & 0 \\
0 & 0 & 0 & 1 & 0 & 0 & 0 & 0 & 0 \\
0 & 0 & 0 & 0 & 0 & 0 & 1 & 0 & 0
\end{array}\right] \right).\]

\noindent \underline{$\vec{n} = (0, 3, 2)$}. Then, we have 
\begin{equation*}
    \begin{split}
    \minimum_0(\Z^{\oplus 2}_{\tilde{*}}) &= \{\{\}\} \\ 
    \minimum_3(\Z^{\oplus 2}_{\tilde{*}}) &= \{\{(1, 0), (0, 1), (-1, 1)\}\} \\ 
    \minimum_2(\Z^{\oplus 2}_{\tilde{*}}) &= \{\{(1, 0), (0, 1)\}\}\\
    \end{split}
\end{equation*}
Now, we will iterate over each triplet  $(\mathcal{X}, \mathcal{Y}, \mathcal{Z}) \in \{\{\}\} \times \{\{(1, 0), (0, 1), (-1, 1)\}\} \times \{\{(1, 0), (0, 1)\}\}$. The only such triplet is  $(\mathcal{X}, \mathcal{Y}, \mathcal{Z}) = (\{\}, \{(1, 0), (0, 1), (-1, 1)\}, \{(1, 0), (0, 1)\})$. Then, 
\[\cK_{\cX,\cY,\cZ} = \cP \left(\left[\begin{array}{rrrrrrrrr}
0 & 0 & 0 & -1 & 1 & 0 & 0 & 0 & 0 \\
0 & 0 & 0 & 1 & 0 & -1 & 0 & 0 & 0 \\
0 & 0 & 0 & 0 & 0 & 2 & 0 & 0 & 0 \\
0 & 0 & 0 & 1 & -1 & 0 & 0 & 0 & 0 \\
0 & 0 & 0 & -1 & 1 & 0 & 0 & 0 & 0 \\
0 & 0 & 0 & 0 & -1 & 1 & 0 & 0 & 0 \\
0 & 0 & 0 & 0 & 1 & -1 & 0 & 0 & 0 \\
0 & 0 & 0 & 0 & 0 & 0 & -1 & 1 & 0 \\
0 & 0 & 0 & 0 & 0 & 0 & 1 & 0 & -1 \\
0 & 0 & 0 & 0 & 0 & 0 & 1 & -1 & 0 \\
0 & 0 & 0 & 0 & 0 & 0 & -1 & 1 & 0
\end{array}\right],  [\quad] \right),\]
and using the optimization described in the previous case, this time we will use a single equality to obtain
\[\cQ_{\cX,\cY,\cZ} = \cP\left(\begin{bmatrix}
0 & 0 & 0 & 1 & 0 & 0 & -1 & 0 & 0 \\
0 & 0 & 0 & -1 & 0 & 0 & 1 & 0 & 0 \\
\end{bmatrix} , [\quad]\right).\]
Here, we can obtain $\cQ_{\cX,\cY,\cZ}$ using the equality $Q_2(\{(1, 0)\}) = Q_3(\{(1, 0)\})$.

Finally, we get the refinement $P'$ as follows, \[P' = P \intersect \cK_{\cX,\cY,\cZ} \intersect \cQ_{\cX,\cY,\cZ},\] which obtains \[P' = \cP\left(\left[\begin{array}{rrrrrrrrr}
-1 & 1 & 0 & 0 & 0 & 0 & 0 & 0 & 0 \\
1 & 0 & -1 & 0 & 0 & 0 & 0 & 0 & 0 \\
0 & 0 & 1 & 0 & 0 & 0 & 0 & 0 & 0 \\
0 & 0 & 0 & 0 & 0 & 0 & -1 & 1 & 0 \\
0 & 0 & 0 & 0 & 0 & 0 & 1 & 0 & -1 \\
0 & 0 & 0 & 0 & 0 & 0 & 0 & 0 & 1 \\
0 & 0 & 0 & -1 & 1 & 0 & 0 & 0 & 0 \\
0 & 0 & 0 & 1 & 0 & -1 & 0 & 0 & 0 \\
0 & 0 & 0 & -1 & 0 & 1 & 0 & 0 & 0 \\
0 & 0 & 0 & 1 & 0 & 0 & -1 & 0 & 0 \\
0 & 0 & 0 & -1 & 0 & 0 & 1 & 0 & 0 \\
0 & 0 & 0 & -1 & -1 & 1 & 1 & 0 & 0 \\
0 & 0 & 0 & 1 & 0 & 0 & 0 & -1 & 0 \\
1 & 0 & 0 & 0 & 0 & 0 & 0 & 0 & 0 \\
0 & 0 & 0 & 1 & 0 & 0 & 0 & 0 & 0 \\
0 & 0 & 0 & 0 & 0 & 0 & 1 & 0 & 0
\end{array}\right], \left[\begin{array}{rrrrrrrrr}
1 & 0 & 0 & 0 & 0 & 0 & 0 & 0 & 0 \\
0 & 0 & 0 & 1 & 0 & 0 & 0 & 0 & 0 \\
0 & 0 & 0 & 0 & 0 & 0 & 1 & 0 & 0
\end{array}\right] \right).\] 

Thus, the refinement of $P$ is given by the collection of all the different $P'$ for different values of $\vec{n} \in \cL_{1, 2}$. 

\medskip
\noindent {\large \textbf{Example 2.}} Suppose we start with a polyhedral cone $P$ with the following description, 
\[P = \cP\left(\left[\begin{array}{rrrrrrrrr}
-1 & 1 & 0 & 0 & 0 & 0 & 0 & 0 & 0 \\
1 & 0 & -1 & 0 & 0 & 0 & 0 & 0 & 0 \\
0 & 0 & 1 & 0 & 0 & 0 & 0 & 0 & 0 \\
0 & 0 & 0 & 0 & 0 & 0 & -1 & 1 & 0 \\
0 & 0 & 0 & 0 & 0 & 0 & 1 & 0 & -1 \\
0 & 0 & 0 & 0 & 0 & 0 & 0 & 0 & 1 \\
0 & 0 & 0 & -1 & 1 & 0 & 0 & 0 & 0 \\
0 & 0 & 0 & 1 & 0 & -1 & 0 & 0 & 0 \\
0 & 0 & 0 & -1 & 0 & 1 & 0 & 0 & 0 \\
0 & 0 & 0 & 1 & 0 & 0 & -1 & 0 & 0 \\
0 & 0 & 0 & -1 & 0 & 0 & 1 & 0 & 0 \\
0 & 0 & 0 & -1 & -1 & 1 & 1 & 0 & 0 \\
0 & 0 & 0 & 1 & 0 & 0 & 0 & -1 & 0 \\
1 & 0 & 0 & 0 & 0 & 0 & 0 & 0 & 0 \\
0 & 0 & 0 & 1 & 0 & 0 & 0 & 0 & 0 \\
0 & 0 & 0 & 0 & 0 & 0 & 1 & 0 & 0
\end{array}\right], \left[\begin{array}{rrrrrrrrr}
1 & 0 & 0 & 0 & 0 & 0 & 0 & 0 & 0 \\
0 & 0 & 0 & 1 & 0 & 0 & 0 & 0 & 0 \\
0 & 0 & 0 & 0 & 0 & 0 & 1 & 0 & 0
\end{array}\right] \right)\]
with covering parameter $((\{\}), (\{(1, 0), (0, 1), (-1, 1)\}), (\{(1, 0), (0, 1)\}))$. We will show an example of refining $P$ for when $\vec{n} = (3, 0, 1)$. Then, 
\begin{equation*}
    \begin{split}
        & \minimum_3(\Z^{\oplus 2}_{\tilde{*}} \setminus \{\}) = \{\{(1, 0), (0, 1), (-1, 1)\}\} \\
        & \minimum_0(\Z^{\oplus 2}_{\tilde{*}} \setminus \{(1, 0), (0, 1), (-1, 1)\}) = \{\{\}\} \\
        & \minimum_1(\Z^{\oplus 2}_{\tilde{*}} \setminus \{(1, 0), (0, 1)\}) = \{\{(-1, 1)\}\} \\
    \end{split}
\end{equation*}
Now, we will iterate over each triplet  $(\mathcal{X}, \mathcal{Y}, \mathcal{Z}) \in  \{\{(1, 0), (0, 1), (-1, 1)\}\} \times \{ \{\} \} \times \{\{(-1, 1)\}\}$.  The only such triplet is  $(\mathcal{X}, \mathcal{Y}, \mathcal{Z}) = (\{(1, 0), (0, 1), (-1, 1)\}, \{\} , \{(-1, 1)\})$. Then, 
\[\cK_{\cX,\cY,\cZ} = \cP \left(\left[\begin{array}{rrrrrrrrr}
-1 & 1 & 0 & 0 & 0 & 0 & 0 & 0 & 0 \\
1 & 0 & -1 & 0 & 0 & 0 & 0 & 0 & 0 \\
0 & 0 & 2 & 0 & 0 & 0 & 0 & 0 & 0 \\
1 & -1 & 0 & 0 & 0 & 0 & 0 & 0 & 0 \\
-1 & 1 & 0 & 0 & 0 & 0 & 0 & 0 & 0 \\
0 & -1 & 1 & 0 & 0 & 0 & 0 & 0 & 0 \\
0 & 1 & -1 & 0 & 0 & 0 & 0 & 0 & 0 \\
0 & 0 & 0 & -1 & 1 & 0 & 0 & 0 & 0 \\
0 & 0 & 0 & 1 & 0 & -1 & 0 & 0 & 0 \\
0 & 0 & 0 & 0 & 0 & 2 & 0 & 0 & 0 \\
0 & 0 & 0 & 1 & -1 & 0 & 0 & 0 & 0 \\
0 & 0 & 0 & -1 & 1 & 0 & 0 & 0 & 0 \\
0 & 0 & 0 & 0 & -1 & 1 & 0 & 0 & 0 \\
0 & 0 & 0 & 0 & 1 & -1 & 0 & 0 & 0 \\
0 & 0 & 0 & 0 & 0 & 0 & -1 & 1 & 0 \\
0 & 0 & 0 & 0 & 0 & 0 & 1 & 0 & -1 \\
0 & 0 & 0 & 0 & 0 & 0 & 0 & 0 & 2 \\
0 & 0 & 0 & 0 & 0 & 0 & 1 & -1 & 0 \\
0 & 0 & 0 & 0 & 0 & 0 & -1 & 1 & 0
\end{array}\right],  [\quad] \right),\]
We will use the equality $Q_1((1, 0)) = Q_3((-1, 1))$ to obtain
\[\cQ_{\cX,\cY,\cZ} = \cP\left(\begin{bmatrix}
1 & 0 & 0 & 0 & 0 & 0 & -1 & -1 & 1 \\
-1 & 0 & 0 & 0 & 0 & 0 & 1 & 1 & -1 \\
\end{bmatrix} , []\right).\]
Finally, we get the refinement $P'$ as follows, \[P' = P \intersect \cK_{\cX,\cY,\cZ} \intersect \cQ_{\cX,\cY,\cZ},\] which obtains \[P' = \cP\left(\left[\begin{array}{rrrrrrrrr}
-1 & 1 & 0 & 0 & 0 & 0 & 0 & 0 & 0 \\
1 & 0 & -1 & 0 & 0 & 0 & 0 & 0 & 0 \\
0 & 0 & 0 & 0 & 0 & 0 & -1 & 1 & 0 \\
0 & 0 & 0 & 0 & 0 & 0 & 1 & 0 & -1 \\
0 & 0 & 0 & 0 & 0 & 0 & 0 & 0 & 1 \\
0 & 0 & 0 & -1 & 1 & 0 & 0 & 0 & 0 \\
0 & 0 & 0 & 1 & 0 & -1 & 0 & 0 & 0 \\
0 & 0 & 0 & -1 & 0 & 1 & 0 & 0 & 0 \\
0 & 0 & 0 & 1 & 0 & 0 & -1 & 0 & 0 \\
0 & 0 & 0 & -1 & 0 & 0 & 1 & 0 & 0 \\
0 & 0 & 0 & -1 & -1 & 1 & 1 & 0 & 0 \\
0 & 0 & 0 & 1 & 0 & 0 & 0 & -1 & 0 \\
-1 & 0 & 1 & 0 & 0 & 0 & 0 & 0 & 0 \\
1 & 0 & 0 & 0 & 0 & 0 & -1 & -1 & 1 \\
-1 & 0 & 0 & 0 & 0 & 0 & 1 & 1 & -1 \\
-1 & -1 & 1 & 0 & 0 & 0 & 1 & 1 & -1 \\
1 & 0 & 0 & 0 & 0 & 0 & 0 & 0 & 0 \\
0 & 0 & 0 & 1 & 0 & 0 & 0 & 0 & 0 \\
0 & 0 & 0 & 0 & 0 & 0 & 1 & 0 & 0
\end{array}\right], \left[\begin{array}{rrrrrrrrr}
1 & 0 & 0 & 0 & 0 & 0 & 0 & 0 & 0 \\
0 & 0 & 0 & 1 & 0 & 0 & 0 & 0 & 0 \\
0 & 0 & 0 & 0 & 0 & 0 & 1 & 0 & 0
\end{array}\right] \right).\] 

\subsection{\texorpdfstring{Explicit Descriptions of $T_0, T_1, T_2, T_3$ and $P_0, P_1, P_2, P_3$}{T0, T1, T2, T3 and P0, P1, P2, P3}}
\label{subsec:p1p2p3p4}

Here, we will give polyhedral cone descriptions of $T_0, T_1, T_2, T_3$ as described in \autoref{sec:a_plus_b_ge_4}.

\[T_0 = \cP\left(\left[\begin{array}{rrrrrrrrr}
-1 & 1 & 0 & 0 & 0 & 0 & 0 & 0 & 0 \\
1 & 0 & -1 & 0 & 0 & 0 & 0 & 0 & 0 \\
0 & 0 & 1 & 0 & 0 & 0 & 0 & 0 & 0 \\
0 & 0 & 0 & 0 & 0 & 0 & -1 & 1 & 0 \\
0 & 0 & 0 & 0 & 0 & 0 & 1 & 0 & -1 \\
0 & 0 & 0 & 0 & 0 & 0 & 0 & 0 & 1 \\
0 & 0 & 0 & -1 & 1 & 0 & 0 & 0 & 0 \\
0 & 0 & 0 & 1 & 0 & -1 & 0 & 0 & 0 \\
0 & 0 & 0 & 0 & 0 & 1 & 0 & 0 & 0 \\
1 & 0 & 0 & 0 & 0 & 0 & 0 & 0 & 0 \\
0 & 0 & 0 & 1 & 0 & 0 & 0 & 0 & 0 \\
0 & 0 & 0 & 0 & 0 & 0 & 1 & 0 & 0
\end{array}\right], \left[\begin{array}{rrrrrrrrr}
1 & 0 & 0 & 0 & 0 & 0 & 0 & 0 & 0 \\
0 & 0 & 0 & 1 & 0 & 0 & 0 & 0 & 0 \\
0 & 0 & 0 & 0 & 0 & 0 & 1 & 0 & 0
\end{array}\right] \right),\]
\[T_1 = \cP\left( \left[\begin{array}{rrrrrrrrr}
-1 & 1 & 0 & 0 & 0 & 0 & 0 & 0 & 0 \\
1 & 0 & -1 & 0 & 0 & 0 & 0 & 0 & 0 \\
0 & 0 & 1 & 0 & 0 & 0 & 0 & 0 & 0 \\
0 & 0 & 0 & 0 & 0 & 0 & -1 & 1 & 0 \\
0 & 0 & 0 & 0 & 0 & 0 & 1 & 0 & -1 \\
0 & 0 & 0 & 0 & 0 & 0 & 0 & 0 & 1 \\
0 & 0 & 0 & -1 & 1 & 0 & 0 & 0 & 0 \\
0 & 0 & 0 & 1 & 0 & -1 & 0 & 0 & 0 \\
0 & 0 & 0 & 0 & 0 & 1 & 0 & 0 & 0 \\
1 & 0 & 0 & 0 & 0 & 0 & -1 & 0 & 0 \\
-1 & 0 & 0 & 0 & 0 & 0 & 1 & 0 & 0 \\
0 & 0 & 0 & 1 & 0 & 0 & -1 & 0 & 0 \\
0 & 0 & 0 & -1 & 0 & 0 & 1 & 0 & 0 \\
1 & 0 & 0 & 0 & 0 & 0 & 0 & 0 & 0 \\
0 & 0 & 0 & 1 & 0 & 0 & 0 & 0 & 0 \\
0 & 0 & 0 & 0 & 0 & 0 & 1 & 0 & 0
\end{array}\right], \left[\begin{array}{rrrrrrrrr}
1 & 0 & 0 & 0 & 0 & 0 & 0 & 0 & 0 \\
0 & 0 & 0 & 1 & 0 & 0 & 0 & 0 & 0 \\
0 & 0 & 0 & 0 & 0 & 0 & 1 & 0 & 0
\end{array}\right] \right),\]
\[T_2 = \cP\left( \left[\begin{array}{rrrrrrrrr}
-1 & 1 & 0 & 0 & 0 & 0 & 0 & 0 & 0 \\
1 & 0 & -1 & 0 & 0 & 0 & 0 & 0 & 0 \\
0 & 0 & 1 & 0 & 0 & 0 & 0 & 0 & 0 \\
0 & 0 & 0 & 0 & 0 & 0 & -1 & 1 & 0 \\
0 & 0 & 0 & 0 & 0 & 0 & 1 & 0 & -1 \\
0 & 0 & 0 & 0 & 0 & 0 & 0 & 0 & 1 \\
0 & 0 & 0 & -1 & 1 & 0 & 0 & 0 & 0 \\
0 & 0 & 0 & 1 & 0 & -1 & 0 & 0 & 0 \\
0 & 0 & 0 & 0 & 0 & 1 & 0 & 0 & 0 \\
1 & 0 & 0 & 0 & 0 & 0 & -1 & 0 & 0 \\
-1 & 0 & 0 & 0 & 0 & 0 & 1 & 0 & 0 \\
0 & 0 & 0 & 1 & 0 & 0 & -1 & 0 & 0 \\
0 & 0 & 0 & -1 & 0 & 0 & 1 & 0 & 0 \\
0 & 1 & 0 & 0 & 0 & 0 & 0 & -1 & 0 \\
0 & -1 & 0 & 0 & 0 & 0 & 0 & 1 & 0 \\
0 & 0 & 0 & 0 & 1 & 0 & 0 & -1 & 0 \\
0 & 0 & 0 & 0 & -1 & 0 & 0 & 1 & 0 \\
1 & 0 & 0 & 0 & 0 & 0 & 0 & 0 & 0 \\
0 & 0 & 0 & 1 & 0 & 0 & 0 & 0 & 0 \\
0 & 0 & 0 & 0 & 0 & 0 & 1 & 0 & 0
\end{array}\right], \left[\begin{array}{rrrrrrrrr}
1 & 0 & 0 & 0 & 0 & 0 & 0 & 0 & 0 \\
0 & 0 & 0 & 1 & 0 & 0 & 0 & 0 & 0 \\
0 & 0 & 0 & 0 & 0 & 0 & 1 & 0 & 0
\end{array}\right]\right),\]
\[T_3 = \cP\left(\left[\begin{array}{rrrrrrrrr}
-1 & 1 & 0 & 0 & 0 & 0 & 0 & 0 & 0 \\
0 & 0 & 0 & 0 & 0 & 0 & -1 & 1 & 0 \\
0 & 0 & 0 & 0 & 0 & 0 & 1 & 0 & -1 \\
0 & 0 & 0 & -1 & 1 & 0 & 0 & 0 & 0 \\
0 & 0 & 0 & 0 & 0 & 1 & 0 & 0 & 0 \\
1 & 0 & 0 & 0 & 0 & 0 & -1 & 0 & 0 \\
-1 & 0 & 0 & 0 & 0 & 0 & 1 & 0 & 0 \\
0 & 0 & 0 & 1 & 0 & 0 & -1 & 0 & 0 \\
0 & 0 & 0 & -1 & 0 & 0 & 1 & 0 & 0 \\
0 & 1 & 0 & 0 & 0 & 0 & 0 & -1 & 0 \\
0 & -1 & 0 & 0 & 0 & 0 & 0 & 1 & 0 \\
0 & 0 & 0 & 0 & 1 & 0 & 0 & -1 & 0 \\
0 & 0 & 0 & 0 & -1 & 0 & 0 & 1 & 0 \\
1 & 1 & -1 & 0 & 0 & 0 & -1 & -1 & 1 \\
-1 & -1 & 1 & 0 & 0 & 0 & 1 & 1 & -1 \\
0 & 0 & 0 & 1 & 1 & -1 & -1 & -1 & 1 \\
0 & 0 & 0 & -1 & -1 & 1 & 1 & 1 & -1 \\
1 & 0 & 0 & 0 & 0 & 0 & 0 & 0 & 0 \\
0 & 0 & 0 & 1 & 0 & 0 & 0 & 0 & 0 \\
0 & 0 & 0 & 0 & 0 & 0 & 1 & 0 & 0
\end{array}\right], \left[\begin{array}{rrrrrrrrr}
1 & 0 & 0 & 0 & 0 & 0 & 0 & 0 & 0 \\
0 & 0 & 0 & 1 & 0 & 0 & 0 & 0 & 0 \\
0 & 0 & 0 & 0 & 0 & 0 & 1 & 0 & 0
\end{array}\right] \right).\]

\medskip
The covering parameters $P_0, P_1, P_2, P_3$ as described in \autoref{sec:a_plus_b_ge_4} are given by 

\[P_0 = ((), (), ()),\]
\[P_1 = ((\{((1, 0))\}), (\{((1, 0))\}), (\{((1, 0))\})),\]
\[P_2 = ((\{(1, 0)\}, \{(0, 1)\}), (\{(1, 0)\}, \{(0, 1)\}), (\{(1, 0)\}, \{(0, 1)\})),\]
\[P_3 = ((\{(1, 0)\}, \{(0, 1)\}, \{(-1, 1)\}), (\{(1, 0)\}, \{(0, 1)\}, \{(-1, 1)\}), (\{(1, 0)\}, \{(0, 1), \{(-1, 1)\}\})).\]

\subsection{Explicit Indexing in \autoref{lemma:succ_min_extended}}
\label{subsec:explicit_indexing}

The indexing procedure from \autoref{lemma:succ_min_extended} as described in  \autoref{eqn:1_to_c1}, \autoref{eqn:c1plus1_to_c1plusc2}, and  \autoref{eqn:c1plusc2plus1_to_c1plusc2plusc3} is  explicitly illustrated as follows:

\[\underbrace{\underbrace{\vec{s}_{ik_i}, \ldots, \vec{s}_{i(k_i + L_{1, i} - 1)}}_{L_{1, i} \text{ elements }}, \ldots, \underbrace{\vec{s}_{i(k_i + (c_1 - 1)L_{1, i})}, \ldots ,\vec{s}_{i(k_i + c_1L_{1, i} - 1)}}_{L_{1, i}\text{ elements }}}_{c_1 \text{ sets}},\]
\[\underbrace{\underbrace{\vec{s}_{i(k_i + c_1L_{1, i})}, \ldots, \vec{s}_{i(k_i + c_1L_{1, i} + L_{2, i} - 1)}}_{L_{2, i}\text{ elements }}, \ldots, \underbrace{\vec{s}_{i(k_i + c_1L_{1, i} + (c_2 - 1)L_{2 , i})}, \ldots ,\vec{s}_{i(k_i + c_1L_{i , 1} + c_2L_{2, i} - 1)}}_{L_{2, i}\text{ elements }}}_{c_2 \text{ sets}},\]
\[\underbrace{\underbrace{\vec{s}_{i(k_i + c_1L_{1, i} + c_2L_{2, i})}, \ldots, \vec{s}_{i(k_i + c_1L_{1, i} + c_2L_{2, i} + L_{3, i} - 1)}}_{L_{3, i}\text{ elements }}, \ldots, \underbrace{\vec{s}_{i(k_i + c_1L_{1, i} + c_2L_{2, i} + (c_3 - 1)L_{3, i})}, \ldots ,\vec{s}_{i(k_i + c_1L_{1, i} + c_2L_{2, i} + c_3L_{3, i} - 1)}}_{L_{3, i}\text{ elements }}}_{c_3 \text{ sets}}.\]

\end{document}